%% file: main.tex
\documentclass[12pt]{article}
\usepackage{amsmath,amssymb,amsthm}
\usepackage{epsfig,graphicx}
\usepackage{enumerate}
\usepackage{url}
\usepackage{color}
\usepackage{srcltx}
\usepackage[mathscr]{eucal}
\usepackage[math]{easyeqn}
\usepackage{etoolbox}
\usepackage{hyperref}
\usepackage{subfig}
\usepackage{array,tabularx,tabulary,booktabs}
\usepackage{longtable}
\usepackage{multirow}
\usepackage{rotating}
\usepackage{floatrow}

\usepackage{nameref}
\usepackage{mathrsfs}
\usepackage{algorithm}
\usepackage[noend]{algpseudocode}
\usepackage{wasysym}
\usepackage{bbm}
\usepackage{enumitem}

\input{contents/ts_def.tex}

\date{}

\begin{document}
	
	\def\spacingset#1{\renewcommand{\baselinestretch}{#1}\small\normalsize} 
	\spacingset{1}
	
		\title{\bf Manifold-based time series forecasting\thanks{
				The article was prepared within the framework of the HSE University 
				Basic Research Program.
				Financial support by the German Research Foundation (DFG) through 
				the Collaborative Research Center 1294 is gratefully acknowledged.
				The results of Section \ref{sec_theoretical} have been obtained under 
				support of the RSF grant No. 19-71-30020.}
		}
		\author{Nikita Puchkin\\
			HSE University and Institute for Information Transmission Problems RAS,\\
			Aleksandr Timofeev \\
			\' Ecole Polytechnique F\' ed\' erale de Lausanne,\\
			and\\
			Vladimir Spokoiny \\
			Weierstrass Institute and Humboldt University,\\
			HSE University and Institute for Information Transmission Problems RAS}
		\maketitle
	
	\bigskip
	\begin{abstract}
		Prediction for high dimensional time series is a challenging task due to the 
		curse of dimensionality problem. Classical parametric models like ARIMA or VAR 
		require strong modeling assumptions and time stationarity and are often 
		overparametrized. This paper offers a new flexible approach using recent ideas 
		of manifold learning. The considered model includes linear models such as the 
		central subspace model and ARIMA as particular cases. The proposed 
		procedure combines manifold denoising techniques with a simple 
		nonparametric prediction by local averaging. The resulting procedure 
		demonstrates a very reasonable performance for real-life econometric time 
		series. We also provide a theoretical justification of the manifold estimation 
		procedure. 
	\end{abstract}
	
	\vfill
	
	\newpage
	\spacingset{1.45} 
	
	\input{contents/ts_introduction.tex}
	\input{contents/ts_model.tex}
	\input{contents/ts_methodology.tex}
	\input{contents/ts_numerical.tex}
	\input{contents/ts_theoretical.tex}
	\input{contents/ts_proofs.tex}
	
	\bibliographystyle{plain}

\input{main.bbl}
	\appendix
	\input{contents/ts_appendix_a.tex}
	\input{contents/ts_appendix_b.tex}
	\input{contents/ts_appendix_c.tex}
	\input{contents/ts_appendix_d.tex}

\end{document}

%% file: contents/ts_def.tex
\def\argmin{\operatornamewithlimits{argmin}}

\def\cond{\, | \,}

\renewcommand\leq\leqslant
\renewcommand\geq\geqslant

\newcommand\eps\varepsilon

\newcommand{\E}{\mathbb E}
\newcommand{\1}{\mathbbm 1}
\newcommand{\p}{\mathbb P}
\newcommand{\q}{\mathbb Q}
\newcommand{\R}{\mathbb R}

\newcommand{\A}{\mathcal A}
\newcommand{\B}{\mathcal B}
\newcommand{\C}{\mathcal C}

\newcommand{\F}{\mathcal F}

\newcommand{\J}{\mathcal J}
\newcommand{\K}{\mathcal K}

\newcommand{\M}{\mathcal M}
\newcommand{\N}{\mathcal N}

\newcommand{\Rad}{\mathcal R}
\newcommand{\T}{\mathcal T}

\newcommand{\bd}{{\boldsymbol D}}
\newcommand{\bi}{{\boldsymbol I}}
\newcommand{\bl}{{\boldsymbol L}}

\newcommand{\bu}{{\boldsymbol U}}
\newcommand{\bw}{{\boldsymbol W}}

\newcommand{\by}{{\boldsymbol Y}}

\newcommand{\bpi}{{\boldsymbol\Pi}}
\newcommand{\bsigma}{{\boldsymbol\Sigma}}

\newcommand\ind[1]{^{(#1)}}

\newcommand{\mclass}{\mathscr M}

\newcommand\dd{\text{\rm d}}
\renewcommand\dim{\text{dim}}
\newcommand\proj[2]{\pi_{#1}\left(#2\right)}
\newcommand\reach[1]{\text{reach}\left(#1\right)}

\newtheorem{Lem}{Lemma}
\newtheorem{Th}{Theorem}

\newtheorem{Ex}{Example}

\renewcommand{\(}{$\,}
\renewcommand{\)}{\,$}

\renewcommand{\Gamma}{\varGamma}
\renewcommand{\Pi}{\varPi}
\renewcommand{\Sigma}{\varSigma}
\renewcommand{\Delta}{\varDelta}
\renewcommand{\Lambda}{\varLambda}
\renewcommand{\Psi}{\varPsi}
\renewcommand{\Phi}{\varPhi}
\renewcommand{\Theta}{\varTheta}
\renewcommand{\Omega}{\varOmega}
\renewcommand{\Xi}{\varXi}
\renewcommand{\Upsilon}{\varUpsilon}

\def\infl{\inf\limits}

\def\argmin{\operatornamewithlimits{argmin}}

\usepackage{color}

\definecolor{blue(pigment)}{rgb}{0.2, 0.2, 0.6}
\definecolor{ultramarine}{rgb}{0.07, 0.04, 0.56}
\definecolor{darkspringgreen}{rgb}{0.09, 0.45, 0.27}
\definecolor{hookersgreen}{rgb}{0.0, 0.44, 0.0}
\definecolor{plum(traditional)}{rgb}{0.56, 0.27, 0.52}
\definecolor{purple(html/css)}{rgb}{0.5, 0.0, 0.5}
\definecolor{magenta(dye)}{rgb}{0.79, 0.08, 0.48}

%% file: contents/ts_introduction.tex
\section{Introduction}

We consider the problem of time series forecasting, which finds plenty of applications in 
different areas such as economics, geology, physics, system 
fault, and planning tasks.
The inability to make a consistent prediction may have a strong influence on 
developing companies and even countries, while a correct forecast can partially 
neutralize the crucial consequences.
For example, it was possible to know in advance about a stock market crash 
\cite{ghazali2007higher}, to forecast a natural disaster \cite{moustra2011artificial}, or 
compute electricity costs for a certain period of time \cite{azadeh2007forecasting}.

Classical forecasting methods 
such as ARMA \cite{10.2307/2332724}, ARIMA (see e.g. \cite{box2015time}), ARFIMA 
\cite{granger1980introduction}, GARCH \cite{bollerslev1986generalized} and LSTM 
\cite{lstm} among many others are based on parametric modeling.
The main drawbacks and limitations of such modeling is that it requires very restrictive parametric assumptions 
and time homogeneity over the whole observation period. 
Modern time series prediction techniques use more flexible nonparametric or semiparametric models
and recent advances in machine learning.
We mention Bayesian methods \cite{de2007bayesian}, Lasso \cite{asw20}, kernel-based prediction \cite{ggh20}, 
deep learning \cite{lz20}, and reinforcement learning \cite{pbtkp20} among many others.
However, for multivariate time series, all the mention approaches suffer from 
\emph{curse of dimensionality} problem:
the used models become quickly overparametrized as the dimension grows.
As a remedy, one or another dimensionality reduction technique assuming that 
the observed high-dimensional time series have 
nevertheless a low-dimensional structure.
Once the low-dimensional structure is recovered, one can make a more accurate 
forecast.
For this purpose, different manifold learning methods and dimension reduction 
techniques can be used.
For instance, usual PCA is useful for linear latent factor models \cite{lyb11}.
For nonlinear latent factor models one can use local PCA \cite{sw02} and kernel PCA 
\cite{ormn17}.
Local linear models were also considered in \cite{llyal06, tmzc15}.
In \cite{psy10}, the authors studied a central subspace model, which is similar to the 
problem of an effective dimension reduction subspace estimation (see e.g. \cite{hjs01b}) 
in i.i.d. setup.
In \cite{rcj18}, the authors used Diffusion maps \cite{c05} to reduce dimensionality.
One can also use another classical dimension reduction methods such as Isomap 
\cite{tenenbaum2000global}, LLE  \cite{roweis2000nonlinear}, LTSA 
\cite{zhang2007linear}, Laplacian Eigenmaps \cite{belkin2002laplacian}, Hessian 
Eigenmaps \cite{donoho2003hessian}, and T-SNE \cite{van2009dimensionality}.

Unfortunately, the mentioned dimension reduction methods assume that the 
observations lie precisely on a smooth manifold and often show poor performance 
when deal with noisy inputs.
One way to overcome this issue is to use functional PCA (FPCA) 
\cite{viviani2005functional, shang2014survey} which directly works with curves instead 
of vectors.
Another approach is to project the data onto a manifold using manifold denoising 
methods \cite{hein2007manifold, wcp10, ldmm, ps19}.
In our work, we assume that the data from a sliding window lie in 
a vicinity of a low-dimensional manifold.
We use recently proposed manifold denoising methods \cite{ldmm} and 
\cite{ps19} to project the data onto the manifold to capture the data 
structure.
After that, we use a weighted k-nearest neighbours predictor which is a very simple but 
yet efficient method for time series forecasting \cite{mfp17, zls17}.
To make the k-NN method exploit the low-dimensional structure, we use the weights 
based on pairwise distances between the projected data points.

For theoretical analysis, we introduce a latent variable model with manifold structure.
There is a vast of literature concerning identification of linear dynamical systems (see, 
for instance, \cite{mdss17, cw18, ftm18, smtjr18, zc18}) but the non-linear case we 
consider is not studied so well.
Our main contribution is that we obtain non-asymptotic upper bounds on accuracy of 
manifold reconstruction in two scenarios.
The first scenario is the case of mixing time series.
There are a lot of papers (e.g. \cite{sc09, mr09, km14, mdss17}) which consider mixing 
time series.
In particular, in \cite{cw18, ftm18}, the authors study stable linear dynamical systems.
The second scenario we consider is the case of non-mixing time series.
In this situation, the analysis is more technically involving.
In \cite{smtjr18}, the authors introduce martingale small ball condition and prove upper 
bounds on the least squares estimator which are also valid for unit root autoregressive 
model.
In \cite{zc18}, the authors extend the results of \cite{smtjr18} to the case of an 
autoregressive model with linear constraints.
In our work, we adapt the martingale small ball condition for non-linear setup.

The rest of this paper is organised as follows.
In Section \ref{sec_model}, we introduce a latent variable model with manifold structure.
In Section \ref{sec_methodology}, we describe our methodology for time series 
forecasting.
Then we present the performance of our method in time series forecasting in Section 
\ref{seс_numerical}.
Finally, in Section \ref{sec_theoretical}, we provide theoretical upper bounds on the 
accuracy of manifold estimation in our model.
Proofs of the main theoretical results can be found in \ref{sec_proofs}, auxiliary 
results are moved to Appendix.

\subsection*{Notations}

Throughout the paper, boldfaced letters are reserved for matrices.
Vectors and scalars are written in regular font.
For any matrix ${\boldsymbol A}$, $\|{\boldsymbol A}\|$ stands for its operator norm, 
and $\|{\boldsymbol A}\|_F$ is the Frobenius norm of ${\boldsymbol A}$.
The notation  $f(n) \lesssim g(n)$ means that there exists an absolute constant $c>0$, 
such that $f(n) \leq c g(n)$ for all $n$.
The relation $f(n) \asymp g(n)$ is equivalent to $f(n) \lesssim g(n)$ and $f(n) \gtrsim 
g(n)$.
Next, for any set $A$ and any $x\in\R^D$, $d(x, A) = \infl_{y\in A} \|x - y\|$ denotes the 
Euclidean distance from $x$ to $A$.

%% file: contents/ts_model.tex
\section{Statistical model}
\label{sec_model}

We assume that we observe a multivariate time series $Y_1, \dots, Y_T \in \R^D$, which 
follows the model
\begin{equation}
	\label{model}
	Y_t = X_t + \eps_t, \quad 1 \leq t \leq T,
\end{equation}
where $X_t$ is a Markov chain on a hidden $d$-dimensional manifold $\M^*$, $d < D$, 
and $\eps_1, \dots, \eps_t$ are independent zero-mean innovations.
We give some examples where a model with a hidden low-dimensional structure appears.

\begin{Ex}[central subspace model, \cite{psy10}]
	Let $g : \R^d \rightarrow \R^p$ be a smooth function and let ${\boldsymbol\Phi}$ be 
	a $(p\times 
	d)$ matrix. The central subspace model is given by the formula
	\[
		Z_t = g({\boldsymbol\Phi}^T Z_{t-1}) + \xi_t, \quad 1 \leq t \leq T.
	\]
	Take $X_t = (Z_{t-1}, g({\boldsymbol\Phi}^T Z_{t-1})) \in \R^{2p}$, $Y_t = (Z_{t-1}, 
	Z_t)$, $\eps_t = (0, \xi_t)$. Then $Y_t = X_t + \eps_t$ and $X_t$ lies on the graph of 
	$g\circ{\boldsymbol\Phi}^T$ which is a $d$-dimensional submanifold in $\R^{D}$ 
	with $D = 2p$.
\end{Ex}

\begin{Ex}
	A natural extension of the central subspace model is
	\[
		Z_t = g(\proj{\M}{Z_{t-1}}) + \xi_t, \quad 1 \leq t \leq T.
	\]
	As before, $g : \R^d \rightarrow \R^D$ is a a smooth function and $\M$ is an 
	unknown smooth $d$-dimensional submanifold in $\R^p$.
	Assume $\M$ is such that there is a global diffeomorphism $\varphi : \R^p 
	\rightarrow \R^d$, which isometrically maps $\M$ into $\R^d$.
	Then for any $z \in \R^p$ there exists $u \in \R^d$ such that $\proj\M z = 
	\varphi^{-1}(u)$.
	Consider $X_t = (Z_{t-1}, g(\proj{\M}{Z_{t-1}})) \in \R^{2p}$, $Y_t 
	= (Z_{t-1}, Z_t)$, $\eps_t = (0, \xi_t)$.
	Then $X_t$ lies on the graph of $g \circ \varphi^{-1}$, which is a $d$-dimensional 
	submanifold in $\R^{D}$ with $D = 2p$.
\end{Ex}

\begin{Ex}[univariate autoregressive model]
	A standard univariate autoregressive model of order $\tau$ is given by
	\[
		Z_t = \sum\limits_{i = 1}^\tau a_i Z_{t-i} + \xi_t, \quad 1 \leq t \leq T.
	\]
	Fix $D > \tau$ and apply a sliding window technique: $Y_t = (Z_t, \dots, Z_{t-D+1}) 
	\in \R^D$.
	Then the autoregressive model can be rewritten as $Y_{t} = {\boldsymbol A} Y_{t-1} 
	+ \eps_t$, where $\eps_t = (\xi_t, 0, \dots, 0) \in \R^D$ and
	\[
		{\boldsymbol A} =
		\begin{pmatrix}
			a_1 & a_2 & \dots & a_k & 0 & \dots & 0 & 0\\
			1 & 0 & \dots & 0 & 0 & \dots & 0 & 0\\
			0 & 1 & \dots & 0 & 0 & \dots & 0 & 0\\
			\vdots & \vdots & \ddots & \vdots & \vdots & \ddots & \vdots & \vdots\\
			0 & 0 & \dots & 1 & 0 & \dots & 0 & 0\\
			0 &  0 & \dots & 0 & 1 & \dots & 0 & 0\\
			\vdots & \vdots & \ddots & \vdots & \vdots & \ddots & \vdots & \vdots\\
			0 &  0 & \dots & 0 & 0 & \dots & 1  & 0\\
		\end{pmatrix}
		\in \R^{D\times D}.
	\]
	In this case, $X_t = {\boldsymbol A} Y_{t-1}$ lies on $Im({\boldsymbol A})$. 
	Since $\text{rank}({\boldsymbol A}) < D$, $Im(A)$ is a linear subspace in $\R^D$ of 
	dimension $\text{rank}({\boldsymbol A})$.
\end{Ex}

We have to impose some regularity conditions on the underlying manifold $\M^*$.
One of the bottlenecks in nonlinear manifold estimation is high curvature of the manifold 
(see, for example, \cite{bml06}).
To overcome this issue, we have to require that $\M^*$ is smooth enough.
We assume that
\begin{align}
	\label{a1}
	\M^{*}\in \mclass_\varkappa^{d}
	= \big\{
	\notag
	\M \subset \R^{D} : \M\text{ is a compact, connected manifold}
	\\\tag{A1}
	\text{without a boundary, } \M\in\C^2, \M \subseteq \B(0, R),
	\\
	\notag
	\text{Vol}(\M) \leq V,
	\reach\M \geq \varkappa, \text{dim}(\M) = d < D \big\}.
\end{align}
The reach of a manifold $\M$ is defined as a supremum of such $r$ that any point $y \in 
\R^D$, such that $d(y, \M) \leq r$, has a unique Euclidean projection onto $\M$.
The assumption \eqref{a1} is ubiquitous in manifold learning literature (see e.g. 
\cite{nm10, gppvw12a, gppvw12b, fmn16, al19}).

We also have to require some properties of the underlying Markov chain $\{X_t : 1 
\leq t \leq T\}$.
It is a common assumption in manifold learning literature (e.g. \cite{gppvw12a, 
gppvw12b, al19, ps19}) that, for any $t$, the marginal density of $X_t$ is bounded away 
from zero.
In our setup, this would imply exponential ergodicity of the Markov chain $\{X_t\}$.
Let $\p_t$ be the marginal distribution of $X_t$ and assume that the Markov chain 
$\{X_t\}$ has a stationary distribution $\pi$.
We require the following: there exist $A > 0$ and $\rho \in (0, 1]$ such that, for any $t \in 
\{1, \dots, T\}$, the measure $\p_t$ satisfies
\begin{equation}
	\label{a2}
	\tag{A2}
	\|\p_{t} - \pi\|_{TV} \leq A^2 (1 - \rho)^t,
\end{equation}
where $\|\cdot\|_{TV}$ is the total variation distance.

Unfortunately, it is not always the case that a latent Markov chain is exponentially mixing.
In our work, besides the case of ergodic Markov chain $\{X_t : 1 
\leq t \leq T\}$, we also consider the situation of non-mixing Markov chain.
The next assumption is a relaxation of \eqref{a2} and admits the absence of mixability.
Let $\F_t$ be a sigma-algebra generated by $X_1, \dots, X_t$ for $t \in \{1, \dots, T\}$ 
and put $\F_0$ the trivial sigma-algebra. 
We require the following: there exist $k \in \mathbb N, h_0 > 0$, and 
$p_1 \geq p_0 > 0$ such that, for any $t \in \{0, \dots, T-k\}$, $h \in (0, h_0)$, and $x \in 
\M^*$, it holds
\begin{equation}
	\label{a3}
	\tag{A3}
	p_0 h^d \leq \frac1k \sum\limits_{j=1}^k \p\left( X_{t + j} \in \B(x, h) \cond \F_t \right) 
	\leq p_1 h^d.
\end{equation}
Assumption \eqref{a3} is similar to the martingale small ball condition introduced in 
\cite{smtjr18} and it is far less restrictive than \eqref{a2}.
In particular, it admits some periodic Markov chains.

Finally, we describe assumptions about $\eps_1, \dots, \eps_T$.
A random vector $\xi \in \R^D$ is called sub-Gaussian with parameter $\sigma^2$ if
\[
	\sup\limits_{\|u\| = 1} \E e^{\lambda u^T(\xi - \E\xi)} \leq e^{\lambda^2 \sigma^2/2}, 
	\quad \forall \lambda \in \R.
\]
We assume that $\eps_1, \dots, \eps_T$ are independent zero-mean sub-Gaussian 
errors:
\begin{equation}
	\label{a4}
	\tag{A4}
	\E \eps_t = 0, \quad \eps_t \in \text{SG}(\sigma_t^2), \quad \forall \, t \in \{1, \dots, T\}.
\end{equation}

In our paper, we study an empirical risk minimizer (ERM)
\begin{equation}
	\label{hatm}
	\widehat\M \in \argmin\limits_{\M \in \mclass_\varkappa^d} \frac1T 
	\sum\limits_{t=1}^T d^2(Y_t, \M).
\end{equation}
In other words, the manifold $\widehat\M$ is the best fitting manifold based on the 
observed data.
We focus on the case when the manifold dimension $d$ is known.
Otherwise, one can add a regularisation term enforcing small dimension of the 
estimated manifold:
\begin{equation}
	\label{hatmpen}
	\widehat\M \in \argmin\limits_{\M \in \cup_{d=1}^{D-1} \mclass_\varkappa^d} \frac1T 
	\sum\limits_{t=1}^T d^2(Y_t, \M) + \lambda \dim(\M).
\end{equation}
The ERM $\widehat\M$ from \eqref{hatm} is a nice object for theoretical study but in 
practice one cannot perform a minimisation over a set of manifolds.
Instead, one can try to approximate the target functional and mimic the ERM manifold.
In Section \ref{sec_methodology}, we discuss two approximation techniques which lead 
to computationally efficient manifold denoising algorithms.

%% file: contents/ts_methodology.tex
\section{Methodology}
\label{sec_methodology}

Assume, we observe a multivariate time series $Z_1, \dots, Z_T \in \R^{p}$ and our goal is 
to make one step ahead forecast $\widehat Z_{T+1}$.
On the first step, we use a sliding window technique for data preprocessing.
We fix an integer $b$, construct a collection of patches $\{ Y_t = (Z_{t-1}, Z_{t-2}, \dots, 
Z_{t-b}) \in \R^{pb} : b+1 \leq t \leq T \}$ and consider a set of pairs $S_T = \{(Y_t, Z_t) : 
b+1 \leq t \leq T\}$.
We assume that high-dimensional vectors $Y_{b+1}, \dots, Y_{T+1} \in \R^D$, $D = bp$, 
lie around a low-dimensional manifold.
We exploit recent advances in manifold learning, described further in this section, to 
project the patches $Y_t$'s onto a manifold in the patch space.
These methods, in fact, mimic the estimate \eqref{hatm} or \eqref{hatmpen} and then 
return the projections $\widehat X_{b+1}, \dots, \widehat X_{T+1}$ of $Y_{b+1}, \dots, 
Y_{T+1}$ onto the manifold $\widehat\M$, respectively.
After that, our forecast is determined by the weighted $k$-nearest neighbors rule:
\[
	\widehat Z_{T+1} = Z_T + \frac{\sum\limits_{t=b+1}^T w_t (Z_t - 
	Z_{t-1})}{\sum\limits_{t=b+1}^T w_t},
\]
where the weights are defined by the formula
\begin{equation}
	\label{knn_weights}
	w_t = e^{-(T+1 - t)/\tau}\K\left( \frac{\|\widehat X_{T+1} - \widehat 
	X_t\|}{h_k}\right),
\end{equation}
where $h_k$ is the $k$-th smallest value amongst $\|\widehat X_{T+1} - \widehat 
X_{b+1}\|, \dots \|\widehat X_{T+1} - \widehat X_T\|$, $\tau > 0$ is a discounting 
parameter and $\K(\cdot)$ is a localising kernel.
In our work, we use Epanechnikov kernel $\K(u) = 3/4 (1 - u^2)_+$ but one can choose 
other kernels.

If one eagers to make several step ahead prediction, i.e. to construct an estimate of 
$Z_{T+m}$, $m > 1$, then he can use a widespread technique.
Sequentially make one step ahead forecasts $\widehat Z_{T+1}, \dots, \widehat 
Z_{T+m}$ adding the new forecast to the data after each step.

\subsection{Manifold learning techniques}

In this section, we briefly describe how to approximate the target functional \eqref{hatm} 
or \eqref{hatmpen} in order to mimic the (penalised) ERM.
In practice, a manifold $\M$ is often associated with a point cloud $\{U_1, \dots, U_T\}$ 
on it.
Given observations $Y_1, \dots, Y_T$, one can use a nonparametric smoothing technique 
to approximate the squared distances $d^2(Y_t, \M)$ using the point cloud:
\[
	d^2(Y_t, \M) \approx \sum\limits_{j=1}^T w_{tj} \|Y_t - U_j\|^2,
\]
where $w_{tj}$, $1 \leq j \leq T$, are some localising weights.
Then
\begin{equation}
	\label{target_approx}
	\sum\limits_{t=1}^T d^2(Y_t, \M) \approx \sum\limits_{t, j=1}^T w_{tj} \|Y_t - U_j\|^2.
\end{equation}

The choice of $w_{tj}$'s plays an important role in performance of an algorithm.
In \cite{ps19}, the authors introduce the algorithm called SAME (from structure-adaptive 
manifold estimation).
For each $t$, they associate $U_t$ with a projection of $Y_t$ onto $\M$ and introduce a 
projector $\bpi_t$ onto the tangent space at the point $U_t$.
Then they take the localising weights of the form
\begin{equation}
	\label{same_weights}
	w_{tj} = \frac1{T h^d} \K_0 \left( \frac{\|\bpi_t(Y_t - Y_j)\|^2}{h^2} \right) \1\left( \|Y_t - 
	Y_j\| < \tau_0 \right),
\end{equation}
where $\K_0 : \R \rightarrow \R_+$ is a smooth kernel, $\int \K_0(t) \dd t = 1$.
Given $\bpi_1, \dots, \bpi_T$, the values of $U_1, \dots, U_T$, minimising the 
approximated target functional \eqref{target_approx}, are equal to
\[
	\widehat U_t = \frac{\sum\limits_{j=1}^T w_{tj} Y_j}{\sum\limits_{j=1}^T w_{tj}},
\]
where the weights $w_{tj}$ are computed according to \eqref{same_weights}.
After that, the obtained values $\widehat U_1, \dots, \widehat U_T$ are used to update 
the projectors $\bpi_1, \dots, \bpi_T$ and then the procedure repeats.
After several iterations, the projection estimates $\widehat X_1, \dots, \widehat X_T$, 
used in \eqref{knn_weights} are put to $U_1, \dots, U_T$.
The pseudocode of SAME is given in Algorithm \ref{same_alg}.

\begin{algorithm}[H]
	\caption{SAME, \cite{ps19}}
	\label{same_alg}
	\begin{algorithmic}[1]
		\State The initial guesses \( \widehat{\bpi}_{1}\ind{0}, \dots, 
		\widehat{\bpi}_{T}\ind{0} \) of projectors onto tangent spaces, dimension of the 
		manifold $d$, the number of iterations \( K + 1 \), an initial bandwidth \( h_0 \), 
		the threshold \( \tau_0 \) and constants \(a > 1\) and \(\gamma > 0\) are given.
		\For{ \( k \) from \( 0 \) to \( K \)}
		\State Compute the weights \( w_{tj}\ind{k} \) according to the formula
		\[
		w_{tj}\ind{k} = \frac1{Th^d} \K_0 \left( \| \widehat{\bpi}_{t}\ind{k} (Y_{t} 
			- Y_{j}) \|^2 / h_k^2 \right) \1 \left( \|Y_{t} - Y_{j} \| \leq \tau_0 \right), \quad 1 
		\leq t, j \leq T.
		\]
		\State Compute the estimates
		\[
		\widehat{U}_{t}\ind{k} = \left(\sum\limits_{j=1}^n w_{tj}\ind{k} 
			Y_{j}\right) \Big/ \left(\sum\limits_{j=1}^n w_{tj}\ind{k}\right), \quad 1 \leq t 
			\leq T.
		\]
		\State If \( k < K \), for each \( T \) from $1$ to $T$, define a set \( \J_{t}\ind{k} = 
		\{ j :  \|\widehat{U}_{j}\ind{k} - \widehat{U}_{t}\ind{k} \| \leq \gamma h_k \} \) and 
		compute the matrices
		\[
		\widehat{\bsigma}_{t} \ind{k} = \sum\limits_{j\in\J_{t}\ind{k}} 
		(\widehat{U}_{j}\ind{k} - \widehat{U}_{t}\ind{k})(\widehat{U}_{j}\ind{k} - 
		\widehat{U}_{i}\ind{k})^T, \quad 1 \leq t \leq T.
		\]
		\State If \( k < K \), for each \( i \) from 1 to n, define \( \widehat{\bpi}_{i}\ind{k+1} 
		\) as a projector onto a linear span of eigenvectors of \( \widehat{\bsigma}_{i} 
		\ind{k} \), corresponding to the largest \( d \) eigenvalues.
		\State If \( k < K \), set \( h_{k+1} = a^{-1} h_k \).
		\EndFor
		\Return the estimates \( \widehat{X}_1 = \widehat{U}_1\ind K, \dots, 
		\widehat{X}_T = \widehat{U}_T\ind K \).
	\end{algorithmic}
\end{algorithm}

The algorithm SAME uses the manifold's dimension $d$ as an input parameter.
If $d$ is not known in advance, one can use \eqref{hatmpen} instead of \eqref{hatm}.
In this case, the squared distances are also approximated according to 
\eqref{target_approx} but the weights $w_{tj}$ are computed as follows:
\[
	w_{tj} = \K_0\left( \frac{\|U_t - U_j\|}h \right).
\]
The main challenge is to deal with the discrete second term.
Fortunately, Lemma 3.1 in \cite{ldmm} helps to overcome this issue.
Let $\mathcal V$ be an open subset in $\mathcal T_{U_t}\M$ such that $0 \in \mathcal 
V$ and let $\mathcal E_t : 
\mathcal V \rightarrow \R^D$ be an exponential map of $\M$ at $U_t$.
Then 
\[
	\dim(\M) = \| \nabla \mathcal E_t(0) \|_F^2, \quad 1 \leq t \leq T.
\]
Here and further, $\|\cdot\|_F$ stands for the Frobenius norm.
Similarly the squared distances, the squared norm of the gradient of the exponential 
map can be approximated according to the formula
\[
	\| \nabla \mathcal E_t(0) \|_F^2 \approx \sum\limits_{j=1}^T 
	w_{tj} \frac{\|U_t - U_j\|^2}{h^2}.
\]
Thus, the dimension of $\M$ can be approximated by
\[
	\dim(\M) \approx \frac{1}{T} \sum\limits_{t, j=1}^T w_{tj} \frac{\|U_t - U_j\|^2}{h^2},
\]
and, for the target functional \eqref{hatmpen}, we have
\begin{equation}
	\label{target_pen_approx}
	\frac1T \sum\limits_{t=1}^T d^2(Y_t, \M) \approx \frac{1}{T} \sum\limits_{t, 
	j=1}^T w_{tj} \|Y_t - U_j\|^2 + \frac{\lambda}{Th^2} \sum\limits_{t, j=1}^T w_{tj} 
	\|U_t - U_j\|^2.
\end{equation}
The algorithm LDMM in \cite{ldmm} uses the split Bregman iteration 
\cite{go09} to find a local minimum of \eqref{target_pen_approx}.
The pseudocode of LDMM is given in Algorithm \ref{ldmm_alg} below.

\begin{algorithm}[H]
	\caption{LDMM, \cite{ldmm}}
	\label{ldmm_alg}
	\begin{algorithmic}[1]
		\State A matrix $\by = (Y_1, \dots, Y_T)^T \in \R^{T\times D}$ of noisy 
		observations and positive numbers $h, \lambda, \mu$ are given.
		\State Initial guess: $\bu\ind 0 = \by \in \R^{T\times D}$, $r\ind 0 = 0 \in 
		\R^{T\times D}$.
		\While{not converge}
		\State Compute the weight matrix $\bw\ind k = \left( w_{tj}\ind k : 1 \leq t, j \leq T 
		\right)$, where
		\[
			w_{tj}\ind k = e^{-\|U_t\ind k - U_j\ind k\|^2/h^2}.
		\]
		\State Compute matrices $\bd\ind k = \text{diag}(d_1\ind k, \dots, d_T\ind k)$ 
		and $\bl\ind k = 
		\bd\ind k - \bw\ind k$, where
		\[
			d_t\ind k = \sum\limits_{j=1}^n w_{tj}\ind k.
		\]
		\State Solve the following linear matrix equation with respect to $V \in 
		\R^{T\times D}$:
		\[
			(\bl\ind k + \mu \bw\ind k) V = \mu \bw\ind k(\bu^t - r^t),
		\]
		\State Update $\bu\ind k$ by solving the least-squares problem
		\[
			\bu\ind{k+1} \in \argmin\limits_{V'} \| \by - V' \|^2_F + 
			\frac{\lambda}{\mu  h^2} \| V - V' + r\ind k \|_F^2,
		\]
		which is given by the formula
		\[
			\bu\ind{k+1} = \left(\by + \frac{\lambda}{\mu h^2} ( V + r\ind k) \right) \Big/ 
			\left(1 + \frac{\lambda}{\mu h^2} \right).
		\]
		\State Update $r$:
		\[
			r\ind{k+1} = r\ind k + V - \bu\ind{k+1}.
		\]
		\State Put $k \gets k + 1$.
		\EndWhile
		\State\Return $\widehat X_1 = U_1\ind k, \dots, \widehat X_T = U\ind k$.
	\end{algorithmic}
\end{algorithm}

%% file: contents/ts_numerical.tex
\section{Numerical Experiments}
\label{seс_numerical}

In this section, we illustrate performance of the algorithms described in Section 
\ref{sec_methodology}.
The algorithms LDMM and SAME are applied sequentially with the weighted nearest 
neighbours method for econometric multivariate time series forecasting.
The algorithms are compared to the weighted nearest neighbours method without the 
manifold reconstruction step and ARIMA.
We release the code with experiments on 
\href{https://github.com/TimofeevAlex/Manifold-based-time-series-forecasting}{GitHub}.
We use the data provided by the Russian Presidential Academy of National Economy 
and Public Administration. 
It represents various quantities characterizing the living standard of the Russian
people from January 1999 to October 2018.
The data contains four components, among which there is a strong correlation between 
the first two and the last two ones, and moreover, it has seasonability. 
As the series becomes multivariate, the sliding window becomes multivariate too, and 
therefore, the data points fed into the LDMM and SAME input consist of four windows 
corresponding to univariate time series.
Thus, the reconstruction of the manifold proceeds at once for all components of the 
series, which certainly provides additional information that allows improving the quality 
of prediction.

The hyperparameters for all algorithms are tuned only for the one step ahead prediction 
simultaneously for all components and then used in all other simulations.
For LDMM, we take the heat kernel $\K_0(t) = e^{-t^2/4}$ with bandwidth 
$h^2 = 0.001$, the hyperparameters $\lambda$ and $\mu$ are 
set to $h^2/7$ and $1500$, respectively, and the number of iterations is $7$.
The width of the sliding window is $11$ and the number of the nearest neighbours in 
ascending order of the component number is $30, 10, 7, 7$.
The algorithm SAME has a different set of hyperparameters.
We use the width of the sliding window equal to $11$, $\tau_0 = 1.0$, the number of 
iterations is set to $21$ and the numbers of nearest neighbours are $9, 21, 21, 15$ for 
the first, the second, the third, and the fourth components, respectively.
The time discount factor $\tau$ (see \eqref{knn_weights}), involved in all weighted k-NN 
based algorithms, is set to $20$. 
For the ARIMA algorithm, the hyperparameters are $(p, d, q) = (6,1,0)$.

The results of prediction are collected in Table \ref{table:multivar_ts}.
Plots of the predictions are shown in Figures 
\ref{fig_1component} -- \ref{fig_4component} in Appendix \ref{plots}.
First, from Table \ref{table:multivar_ts}, one can observe that LDMM and SAME improve 
the predictions of the weighted nearest neighbors method in all the cases.
Second, one can notice that the performance of LDMM and SAME is comparable to 
ARIMA and often is even better.
However, ARIMA does not particularly react to sudden leaps in the series components 
and minimizes the error by tending to its mean.
This behavior should be taken into account by practitioners when choosing an algorithm, 
because forecasting of such jumps may be extremely crucial in certain tasks.

\begin{table}[H]
	\resizebox{\textwidth}{!}{
		\caption{The RMSEs of predictions for multidimensional time series.}
		\label{table:multivar_ts}
		\begin{tabular}{|l|l|l|l|l|l|}
			\hline
			\multicolumn{2}{|l|}{} & \multicolumn{4}{c|}{RMSE $\times 10^3$} \\
			\hline
			\begin{tabular}[c]{@{}l@{}}
				Lookfront (months)
			\end{tabular}
			& Algorithm &  Component 1 & Component 2 & Component 3 & Component 
			4 \\
			\hline
			\multirow{4}{*}{1} & SAME & 1.2 & 1.9 & 1.5 & 1.7 \\
			\cline{2-6}
			& LDMM & 1.6 & 2.4 & 1.5 & 1.4 \\
			\cline{2-6} 
			& k-NN & 1.7 & 3.9 & 2.4 & 1.8 \\
			\cline{2-6}
			& ARIMA & 0.8 & 2.1 & 1.4 & 1.6 \\
			\hline
			\multirow{4}{*}{2} & SAME & 1.5 & 2.5 & 2.1 & 3.9 \\
			\cline{2-6}
			& LDMM & 1.4 & 2.9 & 2.0 & 4.3 \\
			\cline{2-6} 
			& k-NN & 1.8 & 4.4 & 3.0 & 4.3 \\
			\cline{2-6} 
			& ARIMA & 1.8 & 2.5 & 1.9 & 3.8 \\
			\hline
			\multirow{4}{*}{3} & SAME & 1.9 & 3.3 & 2.8 & 5.1 \\
			\cline{2-6}
			& LDMM & 2.8 & 4.0 & 2.9 & 4.9 \\
			\cline{2-6} 
			& k-NN & 3.0 & 7.1 & 4.2 & 5.8 \\
			\cline{2-6} 
			& ARIMA & 1.8 & 3.2 & 2.0 & 4.9 \\
			\hline
			\multirow{4}{*}{4} & SAME & 2.0 & 3.6 & 3.6 & 5.6 \\
			\cline{2-6}
			& LDMM & 3.0 & 3.9 & 3.5 & 5.0 \\
			\cline{2-6}
			& k-NN & 3.1 & 7.2 & 4.5 & 6.0 \\
			\cline{2-6}
			& ARIMA & 2.0 & 3.3 & 2.1 & 5.1 \\
			\hline
		\end{tabular}
	}
\end{table}

%% file: contents/ts_theoretical.tex
\section{Theoretical results}
\label{sec_theoretical}

In this section, we provide theoretical guarantees on the empirical risk minimiser 
\eqref{hatm}.
We are concerned with the question how well the manifold $\widehat\M$, learned from a 
single trajectory $Y_1, \dots, Y_T$, fits other trajectories.
For this purpose, we introduce a trajectory $Y_1', \dots, Y_T'$ which has the same joint 
distribution as $Y_1, \dots, Y_T$ and is independent of $Y_1, \dots, Y_T$.
For any $\M \in \mclass_\varkappa^d$, we characterise its performance by the expected 
average squared distance over the test trajectory $Y_1', \dots, Y_T'$:
\[
	\frac1T \sum\limits_{t=1}^T \E d^2(Y_t', \M).
\]
We are interested in the generalisation ability of ERM $\widehat\M$ and in upper bounds 
on the excess risk
\[
	\frac1T \sum\limits_{t=1}^T \E d^2(Y_t', \widehat \M) - 
	\frac1T \sum\limits_{t=1}^T \E d^2(Y_t', \M^*).
\]
Our first result concerns the case of ergodic hidden Markov chains.
\begin{Th}
	\label{th1}
	Assume \eqref{a1}, \eqref{a2}, and \eqref{a4}.
	Then, for the ERM \eqref{hatm}, it holds that
	\[
		\frac1T \sum\limits_{t=1}^T \E d^2(Y_t', \widehat \M) - \inf\limits_{\M \in 
		\mclass_\varkappa^d} \frac1T \sum\limits_{t=1}^T \E d^2(Y_t', \M)
		\lesssim
		\begin{cases}
			\sqrt{\frac{D \log T}{T \log(1/(1 - \rho))}}, \quad d < 4,\\
			\frac{\sqrt{D} \log^{3/2} T}{\sqrt{T \log(1/(1 - \rho))}}, \quad d = 4,\\
			\left(\frac{\log T}{T \log(1/(1 - \rho))}\right)^{2/d}, \quad d > 4.
		\end{cases}
	\]
\end{Th}
The result of Theorem \ref{th1} improves the results of \cite{nm10} and \cite{fmn16}, 
where the authors obtained the rates $\widetilde O(T^{-1/(d+4)})$ and $\widetilde 
O(T^{-2/(d+4)})$, respectively, in i.i.d. setup.

For the case of non-mixing Markov chains, we provide the following system identification 
result.
\begin{Th}
	\label{th2}
	Assume \eqref{a1}, \eqref{a3}, and \eqref{a4}.
	Assume that the normal and the tangent components of the noise are independent.
	Let $\sigma_1, \dots, \sigma_T$ be such that there exists a constant 
	$c \in (0, 1)$ such that
	\[
		\frac8T \sum\limits_{t=1}^T \sigma_t^2 + 64D\sigma_{\max}^2 \leq \frac{p_0}{8k} 
		\left(\frac{c\varkappa}4 \right)^{d+2}.
	\]
	Then, with probability at least $1 - 8/T$, we have
	\[
		\frac1T \sum\limits_{t=1}^T d^2(X_t, \widehat\M)
		\lesssim \psi_T
		+ \frac{D\left(\sigma_{\max} \sqrt{\log T} \vee (\log T / T)^{1/d} \right)}T 
		\sum\limits_{t=1}^T \sigma_t^2
		+ \frac{\left( \sigma_{\max}^4 \log^2 T \vee (\log T / T)^{4/d} 
		\right)}{\varkappa^2},
	\]
	where $\sigma_{\max} = \max\limits_{1 \leq t \leq T} \sigma_t$ and
	\[
		\psi_T =
		\begin{cases}
			\frac 1T \sqrt{\sum\limits_{t=1}^T \sigma_t^2}, \quad d < 4,\\
			\frac{\log T}T \sqrt{\sum\limits_{t=1}^T \sigma_t^2}, \quad d = 4,\\
			T^{-2/d}\sqrt{\frac{\sum\limits_{t=1}^T \sigma_t^2}T}, \quad d > 4.
		\end{cases}
	\]
\end{Th}

%% file: contents/ts_proofs.tex
\section{Proofs}
\label{sec_proofs}

This section contains proofs of main results.

\subsection{Proof of Theorem \ref{th1}}

From the definition of $\widehat\M$, we have
\begin{align*}
	&
	\frac1T \sum\limits_{t=1}^T \E d^2(Y_t', \widehat \M) - \inf\limits_{\M \in 
	\mclass_\varkappa^d} \frac1T \sum\limits_{t=1}^T \E d^2(Y_t', \M)
	\\&
	\leq \frac1T \sum\limits_{t=1}^T \left( \E d^2(Y_t', \widehat \M) - d^2(Y_t, \widehat 
	\M)\right) - \inf\limits_{\M \in \mclass_\varkappa^d} \frac1T \sum\limits_{t=1}^T \left( 
	\E d^2(Y_t', \M) - d^2(Y_t, \widehat \M) \right)
	\\&
	\leq 2 \E\sup\limits_{\M \in \mclass_\varkappa^d} \frac1T \left| \sum\limits_{t=1}^T 
	\left( d^2(Y_t, \M) - \E d^2(Y_t, \M) \right) \right|.
\end{align*}

The proof of Theorem \ref{th1} is given in three steps.
Let $\overline\E$ be the expectation with respect to $X_1, \eps_1, \dots, X_T, \eps_T$, 
where $X_1$ is generated with respect to the stationary measure $\pi$.
On the first step, we control the discrepancy between $\E\sup\limits_{\M \in 
\mclass_\varkappa^d} \frac1T \left| \sum\limits_{t=1}^T \left( 
d^2(Y_t, \M) - \E d^2(Y_t, \M) \right) \right|$ and \break$\overline\E\sup\limits_{\M \in 
\mclass_\varkappa^d} \frac1T \left| \sum\limits_{t=1}^T \left( d^2(Y_t, \M) - \E d^2(Y_t, \M) 
\right) \right|$.
\begin{Lem}
	\label{lem3}
	\begin{align*}
	\E\sup\limits_{\M \in \mclass_\varkappa^d} \frac1T \left| \sum\limits_{t=1}^T \left( 
	d^2(Y_t, \M) - \E d^2(Y_t, \M) \right) \right|
	&
	\leq \overline\E \sup\limits_{\M \in \mclass_\varkappa^d}
	\frac1T \left| \sum\limits_{t=1}^T \left( d^2(Y_t, \M) - \overline\E d^2(Y_t, \M) 
	\right) \right|
	\\&
	+ \frac{6A}{T\rho} \left(128 \sigma_{\max}^2 D + 16R^2 \right),
	\end{align*}
	where $\sigma_{\max} = \max\limits_{1 \leq t \leq T} \sigma_t$.
\end{Lem}
The proof of Lemma \ref{lem3} is moved to Appendix \ref{lem3_proof}.
If the initial state $X_1$ of the Markov chain is drawn from the distribution $\pi$ then 
$X_1, \dots, X_T$ are identically distributed random elements though still dependent.
Let $K$ be and integer to be specified later and split the set $\{1, \dots, T\}$ into blocks 
$B_1, \dots, B_K$, where
\[
	B_k = \left\{ k, k + K, k + 2K, \dots \right\}, \quad \forall \, k \in \{1, \dots, 
	K\},
\]
and the size of each block is either $\lfloor T/K \rfloor$ or $\lceil T/K \rceil$.
Then we have
\begin{align*}
	&
	\overline\E \sup\limits_{\M \in \mclass_\varkappa^d} \left| \sum\limits_{t=1}^T 
	\left( d^2(Y_t, \M) - \overline\E d^2(Y_t, \M) \right) \right|
	\\&
	\leq \sum\limits_{k=1}^K \overline\E \sup\limits_{\M \in \mclass_\varkappa^d} 
	\left| \sum\limits_{t\in B_k}\left( d^2(Y_t, \M) - \overline\E d^2(Y_t, \M) \right) 
	\right|.
\end{align*}
The following result shows that one can replace $X_t$'s inside one block $B_k$ by i.i.d. 
copies $\widetilde X_1, \dots, \widetilde X_T$ drawn from the stationary distribution 
$\pi$.
\begin{Lem}
	\label{lem4}
	It holds that
	\begin{align*}
		&
		\overline\E \sup\limits_{\M \in \mclass_\varkappa^d} \left| \sum\limits_{t=1}^T 
		\left( d^2(Y_t, \M) - \overline\E d^2(Y_t, \M) \right) \right|
		\\&
		\leq TA (1 - \rho)^K + \sum\limits_{k=1}^K \widetilde\E \sup\limits_{\M \in 
		\mclass_\varkappa^d} \left| \sum\limits_{t\in B_k}\left( d^2(Y_t, \M) - \widetilde\E 
		d^2(Y_t, \M) \right) \right|,
	\end{align*}
	where the expectation $\widetilde \E$ is taken with respect to $\{\widetilde X_t, 
	\eps_t : t \in \B_k\}$ and $\widetilde X_t$, $t \in B_k$, are i.i.d. copies of $X_t$, $t \in 
	B_k$.
\end{Lem}
The proof of Lemma \ref{lem4} is moved to Appendix \ref{lem4_proof}.
Finally, we control
\[
	\E \sup\limits_{\M \in \mclass_\varkappa^d} \left| \sum\limits_{t\in 
	B_k}\left( d^2(Y_t, \M) - \widetilde\E d^2(Y_t, \M) \right) \right|
\]
for each block $B_k$.
\begin{Lem}
	\label{lem5}
	Assume that $B_k = \{t_1, \dots, t_{N_k}\}$, where $N_k$ is the cardinality of $B_k$.
	Then, for any $k \in \{1, \dots, K\}$, it holds that
	\[
		\widetilde \E \sup\limits_{\M \in \mclass_\varkappa^d} \left| \sum\limits_{t \in B_k} 
		d^2 (Y_t, \M) - \widetilde \E d^2 (Y_t, \M) \right|
		\lesssim
		\begin{cases}
		D\sqrt{\sum\limits_{t\in B_k} \sigma_t^2} + R\sqrt{D N_k}, \quad d < 4,\\
		\left(D\sqrt{\sum\limits_{t\in B_k} \sigma_t^2} + R\sqrt{D N_k}\right) \log 
		\psi_k^{-1}, \quad d = 4,\\
		\left(D\sqrt{\sum\limits_{t\in B_k} \sigma_t^2} + R\sqrt{D N_k}\right) 
		\psi_k^{1-4/d}, \quad d > 4,
		\end{cases}
	\]
	where
	\[
		\psi_k = \frac{R\sqrt{D N_k} + D \sqrt{\sum\limits_{t \in B_k} \sigma_t^2}}{R N_k + 
		\sqrt{D} \sum\limits_{t \in B_k} \sigma_t}.
	\]
\end{Lem}
Proof of Lemma \ref{lem5} can be found in Appendix \ref{lem5_proof}.
Take $K = \lceil\log(1/TA) / \log(1 - \rho) \rceil$.
Then
\[
	\widetilde \E \sup\limits_{\M \in \mclass_\varkappa^d} \left| \sum\limits_{t \in B_k} 
	d^2 (Y_t, \M) - \widetilde \E d^2 (Y_t, \M) \right|
	\lesssim
	\begin{cases}
		\sqrt{\frac{D T \log(1/(1 - \rho))}{\log T}}, \quad d < 4,\\
		\sqrt{D T \log(1/(1 - \rho))\log T}, \quad d = 4,\\
		\left(\frac{T \log(1/(1 - \rho))}{\log T}\right)^{1 - 2/d}, \quad d > 4,
	\end{cases}
\]
and the claim of Theorem \ref{th1} follows from Lemmata \ref{lem3}, \ref{lem4}, and 
\ref{lem5}.

\subsection{Proof of Theorem \ref{th2}}

For any $t \in \{1, \dots, T\}$, let $\bpi_t$ be the projector onto $\T_{X_t}\M^*$.
Denote $\eps_t\parallel = \bpi_t \eps_t$ and $\eps_t^\perp = (\bi - \bpi_t)\eps_t$.
Then, for all $\M \in \mclass_\varkappa^d$, it holds that
\begin{equation}
	\label{th2_1}
	d^2(Y_t, \M) = d^2(X_t + \eps_t^\parallel, \M) + 2 a_{\M, t}^T \eps_t^\perp + 
	\|\eps_t^\perp\|^2, \quad \forall \, t 
	\in \{1, \dots, T\},
\end{equation}
where $a_{\M, t} = X_t + \eps_t^\parallel - \proj\M{X_t + \eps_t^\parallel}$.
On the other hand, due to the Cauchy-Schwartz inequality, we have
\begin{equation}
	\label{th2_2}
	d^2(Y_t, \M^*) \leq \left(1 + h^{-1}\right) d^2(X_t + \eps_t^\parallel, \M^*) + (1 + h) 
	\|\eps_t^\perp\|^2, \quad \forall \, t \in \{1, \dots, T\},
\end{equation}
where $h$ is a parameter to be specified later.
The inequalities \eqref{th2_1} and \eqref{th2_2} yield
\[
	d^2(X_t + \eps_t^\parallel, \M)
	\leq 2 a_{\M, t}^T \eps_t^\perp + d^2(Y_t, \M) - d^2(Y_t, \M^*)
	+ (1 + h^{-1}) d^2(X_t + \eps_t^\parallel, \M^*) +  h\|\eps_t^\perp\|^2.
\]
The fact that the reach of $\M^*$ is not less than $\varkappa$ implies that a sphere of 
radius $\varkappa$ rolls freely over the surface of $\M^*$.
Thus, if $\|\eps_t^\parallel\| \leq \varkappa$, we have $d(X_t + \eps_t^\parallel, \M^*) \leq 
2\|\eps_t^\parallel\|^2/\varkappa$.
Then, since $\M^* \subset \B(0, R)$, it holds that
\[
	d(X_t + \eps_t^\parallel, \M^*) \leq \frac{2\|\eps_t^\parallel\|^2}\varkappa + 2R \, 
	\1\left( \|\eps_t^\parallel\| \geq \varkappa \right)
\]
and we obtain
\begin{align}
	\label{th2_3}
	\sum\limits_{t=1}^T d^2(X_t + \eps_t^\parallel, \M)
	&
	\leq 2\sum\limits_{t=1}^T a_{\M, t}^T \eps_t^\perp + \sum\limits_{t=1}^T \left( d^2(Y_t, 
	\M) - d^2(Y_t, \M^*) \right)
	\\&\notag
	+ \sum\limits_{t=1}^T  \left((1 + h^{-1}) \left( \frac{2\|\eps_t^\parallel\|^2}\varkappa + 
	2R \, \1\left( \|\eps_t^\parallel\| \geq \varkappa \right) \right)^2 +  h\|\eps_t^\perp\|^2 
	\right).
\end{align}

Let $\widetilde \bpi_t^*$ and $\widetilde \bpi_t^\M$ be the projectors onto 
$\T_{X_t}\M^*$ and $\T_{\proj\M{X_t}}\M$, respectively.
Then
\[
	\left| d(X_t + \eps_t^\parallel, \M) - d(X_t, \M) \right|
	\leq \|\widetilde\bpi_t^* - \widetilde\bpi_t^\M\| \|\eps_t^\parallel\| + 
	\frac{2\|\eps_t^\parallel\|^2}\varkappa + 
	2R \1\left( \|\eps_t^\parallel\| \geq \varkappa \right).
\]
Then, for any $t \in \{1, \dots, T\}$, we have
\begin{align}
	\label{th2_4}
	d^2(X_t, \M)
	\leq 2d^2(X_t + \eps_t^\parallel, \M) + 2\left(\|\widetilde\bpi_t^* - 
	\widetilde\bpi_t^\M\| \|\eps_t^\parallel\| 
	+ \frac{2\|\eps_t^\parallel\|^2}{\varkappa} + 2R \1\left( \|\eps_t^\parallel\| \geq 
	\varkappa \right)\right)^2.
\end{align}

\begin{Lem}
	\label{lem8}
	Assume that there exists a constant $c \in (0, 1)$ such that
	\[
		\frac8T \sum\limits_{t=1}^T \sigma_t^2 + 64D\sigma_{\max}^2 \leq \frac{p_0}{8k} 
		\left(\frac{c\varkappa}4 \right)^{d+2}.
	\]
	Then, for the ERM $\widehat\M$, defined in \eqref{hatm}, it holds that
	\[
		\p\left( \max\limits_{x\in \M^*} d(x, \widehat\M) \geq c\varkappa \right)
		\leq \frac{4^d V}{(c\varkappa)^d} e^{-\frac{p_0 (c\varkappa/4)^d}{12} \lfloor 
		T/k\rfloor} + \frac1T.
	\]
\end{Lem}
The proofs of auxilary results related to the proof of Theorem \ref{th2} are moved to 
Appendix \ref{th2_proofs}. We assume that $T$ is large enough, so it holds $\frac{4^d 
V}{(c\varkappa)^d} 
e^{-\frac{p_0 (c\varkappa/4)^d}{12} \lfloor T/k\rfloor} \leq 1/T$.

From now on, we can restrict our attention on the event when $\widehat \M \subset \M^* 
+ \B(0, c\varkappa)$.
Lemma \ref{lem8} guarantees that the probability of this event is close to $1$.
Let $\{Z_j^* \in \M : 1 \leq j \leq N\}$ be the maximal $(2h)$-packing on $\M^*$, where 
$h$ is a parameter to be specified later.
Split the manifold $\M^*$ into $N$ disjoint subsets $\{ A_j : 1 \leq j \leq N\}$, such 
that $\B(Z_j^*, h) \subseteq A_j \subseteq \B(Z_j^*, 2h)$.
The existence of such partition follows from the fact that, on one hand, for any $i \neq j$, 
the balls $\B(Z_i^*, h)$ and $\B(Z_j^*, h)$ do not intersect and, on the other hand, for 
any $x \in \M^*$ there exists $j \in \{1, \dots, N\}$ such that $x \in \B(Z_j^*, 2h)$.
For any $x \in \M^*$ and $r > 0$, denote $F_{r}(x) = \B(x, r) \cap (\{x\} + \T_x^\perp\M^*)$ 
and $F_j = \cup_{x\in A_j} F_{c\varkappa}(x)$.
Thus, $\M^* + \B(0, c\varkappa)$ is the union of the disjoint sets $F_j$, $1 \leq j \leq N$.
Fix any $j \in \{1, \dots, N\}$.
For each $\M \in \mclass_\varkappa^d$, we can construct its locally linear approximation.
Namely, let $Z_j^\M \in \M \cap F_j$ be such that the projection of $Z_j^\M$ is equal to 
$Z_j^*$ and denote the projector onto the tangent space $\T_{Z_j^\M}\M$ by $\bpi_j^\M$.
We write $\bpi_j^*$ instead of $\bpi_j^{\M^*}$ for brevity.

Consider any $\M \in \mclass_\varkappa^d$.
Due to Theorem 4.18 in \cite{f59}, for any $j \in \{1, \dots, N\}$ and 
for any $X_t \in A_j$, we have $d_H\left( \M\cap F_j, (\{Z_j^\M\} + \T_{Z_j^\M}\M) \cap F_j 
\right) \leq 2h^2/\varkappa$.
Then it holds that
\[
	\left| d(X_t, \M) - d(X_t, \{Z_j^\M\} + \T_{Z_j^\M}\M) \right| \leq d_H\left( \M\cap F_j, 
	(\{Z_j^\M\} + 
	\T_{Z_j^\M}\M) \cap F_j \right) \leq \frac{C h^2}\varkappa
\]
for some absolute constant $C$.
Using the Cauchy-Schwartz inequality and Theorem 4.18 in \cite{f59}, we obtain
\begin{align*}
	&
	d^2(X_t, \M)
	\geq \frac12 d^2(X_t, \{Z_j^\M\} +\T_{Z_j^\M}\M) - \frac{C^2h^4}{\varkappa^2}
	\\&
	\geq \frac14 d^2(\proj{\{Z_j^*\} +\T_{Z_j^*}\M^*}{X_t}, \{Z_j^\M\} +\T_{Z_j^\M}\M) 
	- \frac{(C^2 + 2) h^4}{\varkappa^2}
	\\&
	= \frac14 d^2(Z_j^* +\bpi_j^*(X_t - Z_j^*), \{Z_j^\M\} +\T_{Z_j^\M}\M) 
	- \frac{(C^2 + 2) h^4}{\varkappa^2}
	\\&
	\gtrsim \|(\bi - \bpi_j^\M)\bpi_j^* (X_t - Z_j^*)\|^2 - \frac{h^4}{\varkappa^2}.
\end{align*}
Using Lemma 3.5 in \cite{blw18}, we obtain
\begin{align*}
	&
	\sum\limits_{X_t \in A_j} \|\widetilde\bpi_t^* - \bpi_j^\M\|^2
	\lesssim \sum\limits_{X_t \in A_j} \|\bpi_j^* - \bpi_j^\M\|^2 + \frac{h^2}{\varkappa^2}
	\\&
	= \sum\limits_{X_t \in A_j} \|(\bi - \bpi_j^\M) \bpi_j^*\|^2 + \frac{h^2}{\varkappa^2}.
\end{align*}
Let $u_1, \dots, u_d$ be an orthonormal basis in $\T_{Z_j^*}\M^*$.
Then
\[
	\|(\bi - \bpi_j^\M) \bpi_j^*\|^2
	\leq d \max\limits_{u \in \{u_1, \dots, u_d\}} \|(\bi - \bpi_j^\M) \bpi_j^* u \|^2.
\]
We prove an upper bound for the right hand side using the following result.
\begin{Lem}
	\label{lem7}
	Assume that $\{ X_t : 1 \leq t \leq T\} \subset \M^*$ satisfies \eqref{a3} and let $h \leq 
	2\varkappa$.
	Then it holds that
	\[
		\p\left( \exists \, x \in \M^* : \sum\limits_{t=1}^T \1\left( X_t \in \B(x, 2h) \right) < 
		(p_0 h^d) \lfloor T/k\rfloor \right)
		\leq \frac{2^d V}{h^d} e^{-\frac{p_0h^d}{12} \lfloor T/k\rfloor}.
	\]
\end{Lem}
We will choose $h \gtrsim (k\log T / T)^{1/d}$ with a sufficiently large hidden constant, so 
we assume that $\frac{2^d V}{h^d} e^{-\frac{p_0h^d}{12} \lfloor T/k\rfloor} < 1/T$.
Due to Lemma \ref{lem7}, with probability at least $1 - 1/T$, for each $u \in \{u_1, \dots, 
u_d\}$ there are $\gtrsim p_0 h^d \lfloor T/k \rfloor$ points amongst $\{X_1, \dots, X_T\} 
\cap A_j$ such that
\[
	\|(\bi - \bpi_j^\M) \bpi_j^* u \|^2 h^2 \lesssim \|(\bi - \bpi_j^\M) \bpi_j^* (X_t - Z_j)\|^2, 
	\quad \forall \, u \in \{u_1, \dots, u_d\}.
\]
Thus,
\[
	p_0 h^d \left\lfloor \frac Tk \right\rfloor \|(\bi - \bpi_j^\M) \bpi_j^*\|^2 h^2
	\lesssim \sum\limits_{X_t \in A_j} d^2(X_t, \M).
\]
Since, according to Lemma \ref{lem7}, $A_j$ contains at most $3 p_1 (2h)^d T / 2$ 
points, we have
\begin{align*}
	&
	\sum\limits_{X_t \in A_j} \|\widetilde\bpi_t^* - \widetilde\bpi_t^\M\|^2
	\lesssim \sum\limits_{X_t \in A_j} \|\widetilde\bpi_t^* - \widetilde\bpi_j^\M\|^2 + 
	\frac{h^2}{\varkappa^2}
	\\&
	\lesssim \sum\limits_{X_t \in A_j} \|\bpi_j^* - \widetilde\bpi_j^\M\|^2 + 
	\frac{h^2}{\varkappa^2}
	\lesssim \frac{p_0}{k p_1} \sum\limits_{X_t \in A_j} d^2(X_t, \M).
\end{align*}
Plugging this bound into \eqref{th2_3} and \eqref{th2_4}, we obtain that
\begin{align*}
	d^2(X_t, \M)
	&
	\lesssim \sum\limits_{t=1}^T \frac{p_0}{k p_1} \cdot \frac{\|\eps_t^\parallel\|^2 
	d^2(X_t, \M)}{h^2} 
	+ \sum\limits_{t=1}^T \left( d^2(Y_t, \M) - d^2(Y_t, \M^*) \right)
	\\&
	+ \sum\limits_{t=1}^T a_{\M, t}^T \eps_t^\perp
	+ \sum\limits_{t=1}^T  \left( \frac{\|\eps_t^\parallel\|^4}{h\varkappa^2} + 
	R^2 \, \1\left( \|\eps_t^\parallel\| \geq \varkappa \right) +  h\|\eps_t^\perp\|^2 + 
	\frac{h^4}{\varkappa^2} \right)
\end{align*}
on an event with probability at least $1 - 3/T$.
Plug the ERM $\widehat\M$ into the last expression:
\begin{align*}
	\sum\limits_{t=1}^T d^2(X_t, \widehat\M)
	&
	\lesssim \sum\limits_{t=1}^T \frac{p_0}{k p_1} \frac{\|\eps_t^\parallel\|^2 d^2(X_t, 
	\widehat\M)}{h^2} + \sup\limits_{\M\in\mclass_\varkappa^d} \sum\limits_{t=1}^T 
	a_{\M, t}^T \eps_t^\perp
	\\&
	+ \sum\limits_{t=1}^T  \left( \frac{\|\eps_t^\parallel\|^4}{h\varkappa^2} + 
	R^2 \, \1\left( \|\eps_t^\parallel\| \geq \varkappa \right) +  h\|\eps_t^\perp\|^2 + 
	\frac{h^4}{\varkappa^2} \right).
\end{align*}

The following large deviation inequalities will be useful.
\begin{Lem}
	\label{lem_conc}
	Let $\delta \in (0, 1)$.
	The following inequalities hold with probability at least $1 - \delta$:
	\begin{enumerate}
		\item[1)] $\max\limits_{1 \leq t \leq T} \|\eps_t^\parallel\| \leq 4\sigma_{\max} 
		\sqrt{d} + 2\sigma_{\max}\sqrt{2 \log(T/\delta)}$, where $\sigma_{\max} 
		=\max\limits_{1 \leq t \leq T} \sigma_t$.
		\item[2)] $\sum\limits_{t=1}^T \|\eps_t^\perp\| \leq 4\sqrt{(D-d)T 
		\sum\limits_{t=1}^T \sigma_t^2} + 2\sqrt{2\log(1/\delta)\sum\limits_{t=1}^T 
		\sigma_t^2}$.
		\item[3)] $\frac1T \sum\limits_{t=1}^T \|\eps_t^\parallel\|^2 \leq \frac4T 
		\sum\limits_{t=1}^T \sigma_t^2 + 64d \sigma_{\max}^2 + \frac{32}T 
		\sqrt{\sum\limits_{t=1}^T \sigma_t^4}\left( \log\frac1\delta 
		\vee \sqrt{\log\frac1\delta} \right)$.
		\item[4)] $\frac1T \sum\limits_{t=1}^T \|\eps_t^\perp\|^2 \leq \frac4T 
		\sum\limits_{t=1}^T \sigma_t^2 + 64(D-d) \sigma_{\max}^2 + \frac{32}T 
		\sqrt{\sum\limits_{t=1}^T \sigma_t^4}\left( \log\frac1\delta 
		\vee \sqrt{\log\frac1\delta} \right)$.
		\item[5)] $\frac1T \sum\limits_{t=1}^T \1\left( \|\eps_t^\parallel\| \geq c\varkappa 
		\right) < \frac{2 \cdot 6^d}T \sum\limits_{t=1}^T 
		e^{-\frac{c^2\varkappa^2}{4\sigma_t^2}} + \frac{2\log(1/\delta)}{T}$.
	\end{enumerate}
\end{Lem}

Choose $h \asymp  \sigma_{\max} \sqrt{\frac{kp_1}{p_0}} (\sqrt{d} + \sqrt{ \log T}) \vee 
(\log T / T)^{1/d}$.
Then it holds that
\begin{align*}
	\frac1T \sum\limits_{t=1}^T d^2(X_t, \widehat\M)
	&
	\lesssim \frac1T \sup\limits_{\M\in\mclass_\varkappa^d} \sum\limits_{t=1}^T a_{\M, 
	t}^T \eps_t^\perp
	+ \frac{D \left(\sigma_{\max} \sqrt{\log T} \vee (\log T / T)^{1/d} \right)}T 
	\sum\limits_{t=1}^T \sigma_t^2
	\\&
	+ \frac{\left( \sigma_{\max}^4 \log^2 T \vee (\log T / T)^{4/d} \right)}{\varkappa^2}
\end{align*}
with probability at least $1 - 7/T$.

It remains to bound the first term in the right hand side.
Note that, due to the conditions of the theorem, for any fixed $\M \in 
\mclass_\varkappa^d$, the vectors $\eps_1^\perp, \dots, \eps_T^\perp$ are independent 
of $a_{\M, 1}, \dots, a_{\M, T}$.
Then $\left( \sum\limits_{t=1}^T a_{\M, t}^T \eps_t^\perp \cond X_1, \eps_1^\parallel, \dots, 
X_T, \eps_T^\parallel \right)$ is a sub-Gaussian process.
Moreover, for any $\M, \M' \in \mclass_\varkappa^d$, $\left( \sum\limits_{t=1}^T a_{\M, 
t}^T \eps_t^\perp \cond X_1, \eps_1^\parallel, \dots, X_T, \eps_T^\parallel \right)$ is a 
sub-Gaussian random variable with variance proxy
\[
	\sum\limits_{t=1}^T \|a_{\M, t} - a_{\M', t}\|^2 \sigma_t^2
	\leq \sum\limits_{t=1}^T d_H^2(\M, \M') \sigma_t^2
	\leq \sum\limits_{t=1}^T 4R^2\sigma_t^2.
\]
Here we used the fact that
\[
	\|a_{\M, t} - a_{\M', t}\|
	= \|\proj\M{X_t + \eps_t^\parallel} - \proj{\M'}{X_t + \eps_t^\parallel}\|
	\leq d_H(\M, \M').
\]
This also yields
\begin{align*}
	&
	\E \sup\limits_{\M, \M'\in\mclass_\varkappa^d, d_H(\M, \M') \leq \gamma} 
	\sum\limits_{t=1}^T (a_{\M, t} - a_{\M', t})^T 
	\eps_t^\perp
	\leq \E \sup\limits_{\M, \M'\in\mclass_\varkappa^d, d_H(\M, \M') \leq \gamma} 
	\sum\limits_{t=1}^T \|a_{\M, t} - a_{\M', t}\| \|\eps_t^\perp\|
	\\&
	\leq \E \sup\limits_{\M, \M'\in\mclass_\varkappa^d, d_H(\M, \M') \leq \gamma} 
	\sum\limits_{t=1}^T d_H(\M, \M') \|\eps_t^\perp\|
	\leq \E \gamma\sum\limits_{t=1}^T \|\eps_t^\perp\|
	\lesssim \gamma \sum\limits_{t=1}^T \sigma_t \sqrt{D-d}.
\end{align*}
Using the chaining technique, we obtain
\begin{align*}
	\E \sup\limits_{\M\in\mclass_\varkappa^d} \sum\limits_{t=1}^T a_{\M, t}^T 
	\eps_t^\perp
	\lesssim \gamma \sum\limits_{t=1}^T \sigma_t \sqrt{D-d} + R 
	\sqrt{\sum\limits_{t=1}^T \sigma_t^2} \int\limits_\gamma^{2R} 
	\sqrt{\log \N(\mclass_\varkappa^d, d_H, \eps)} d\eps .
\end{align*}
Theorem 9 in \cite{gppvw12a} claims that
\[
	\N(\mclass_\varkappa^d, d_H, u)
	\leq c_1 \binom{D}{d}^{(c_2/\varkappa)^D} \exp\left\{ 2^{d/2}(D - d) (c_2 / 
	\varkappa)^D u^{-d/2} \right\},
\]
where the constant $c_2$ depends only on $\varkappa$ and $d$.
This yields
\begin{align*}
	\E \sup\limits_{\M\in\mclass_\varkappa^d} \sum\limits_{t=1}^T a_{\M, t}^T \eps_t^\perp
	&
	\lesssim \gamma \sum\limits_{t=1}^T \sigma_t \sqrt{D-d} + \frac{ 2^{d/4} R 
	\sqrt{D-d}}{\varkappa^{D/2}} \sqrt{\sum\limits_{t=1}^T \sigma_t^2} 
	\int\limits_{\gamma}^{2R} u^{-d/4} du
	\\&
	\lesssim \gamma \sqrt{(D-d)T} \sqrt{\sum\limits_{t=1}^T \sigma_t^2} + \frac{ 2^{d/4} 
	R \sqrt{D-d}}{\varkappa^{D/2}} \sqrt{\sum\limits_{t=1}^T \sigma_t^2} 
	\int\limits_{\gamma}^{2R} u^{-d/4} du.
\end{align*}
Choosing
\[
	\gamma =
	\begin{cases}
		0, \quad d < 4,\\
		\frac{\log T}{\sqrt T}, \quad d = 4,\\
		T^{-2/d}, \quad d > 4,
	\end{cases}
\]
we obtain
\begin{equation*}
	\frac1T \E \sup\limits_{\M\in\mclass_\varkappa^d} \sum\limits_{t=1}^T a_{\M, t}^T 
	\eps_t^\perp
	\lesssim
	\begin{cases}
		\frac 1T \sqrt{(D - d)\sum\limits_{t=1}^T \sigma_t^2}, \quad d < 4,\\
		\frac{\log T}T \sqrt{(D - d)\sum\limits_{t=1}^T \sigma_t^2}, \quad d = 4,\\
		T^{-2/d}\sqrt{\frac{(D - d)\sum\limits_{t=1}^T \sigma_t^2}T}, \quad d > 4.
	\end{cases}
\end{equation*}
Finally, Azuma-Hoeffding inequality yields that
\[
	\sup\limits_{\M\in\mclass_\varkappa^d} \sum\limits_{t=1}^T a_{\M, t}^T \eps_t^\perp
	- \E \sup\limits_{\M\in\mclass_\varkappa^d} \sum\limits_{t=1}^T a_{\M, t}^T 
	\eps_t^\perp
	\lesssim R\sqrt{(D-d) \sum\limits_{t=1}^T \sigma_t^2 \log(T)}
\]
with probability at least $1 - 1/T$.
Thus, with probability at least $1 - 8/T$, it holds that
\[
	\frac1T \sum\limits_{t=1}^T d^2(X_t, \widehat\M)
	\lesssim \psi_T
	+ \frac{D\left(\sigma_{\max} \sqrt{\log T} \vee (\log T / T)^{1/d} \right)}T 
	\sum\limits_{t=1}^T \sigma_t^2
	+ \frac{\left( \sigma_{\max}^4 \log^2 T \vee (\log T / T)^{4/d} \right)}{\varkappa^2},
\]
where
\[
	\psi_T =
	\begin{cases}
		\frac 1T \sqrt{\sum\limits_{t=1}^T \sigma_t^2}, \quad d < 4,\\
		\frac{\log T}T \sqrt{\sum\limits_{t=1}^T \sigma_t^2}, \quad d = 4,\\
		T^{-2/d}\sqrt{\frac{\sum\limits_{t=1}^T \sigma_t^2}T}, \quad d > 4.
	\end{cases}
\]

%% file: contents/ts_appendix_a.tex
\section{Technical tools}

\begin{Lem}
	\label{lem1}
	Let $\p$ and $\q$ be probability measures on a set $\mathcal X$ and let $f : 
	\mathcal X \rightarrow \R$ be a function such that the second moments $\E_\p f^2$ 
	and $\E_\q f^2$ with respect to 
	$\p$ and $\q$ are finite.
	Then
	\[
	\E_\p f - \E_\q f \leq \sqrt{(\E_\p f^2 + \E_\q f^2) \|\p - \q\|_{TV}}.
	\]
\end{Lem}

\begin{proof}
	The claim of Lemma \ref{lem1} follows from the Cauchy-Schwartz inequality
	\begin{align*}
	&
	\left| \E_\p f - \E_\q f \right|
	\leq \int\limits_{\mathcal X} |f(x)| \cdot |\dd\p(x) - \dd\q(x)|
	\\&
	\leq \sqrt{\int\limits_{\mathcal X} f^2(x) \cdot |\dd\p(x) - \dd\q(x)|} 
	\sqrt{\int\limits_{\mathcal X} |\dd\p(x) - \dd\q(x)|}
	\leq  \sqrt{(\E_\p f^2 + \E_\q f^2) \|\p - \q\|_{TV}}.
	\end{align*}
\end{proof}

\begin{Lem}
	\label{lem2}
	Let $\M \in \mclass_\varkappa^d$ and let $\eps_t$ be a sub-Gaussian random 
	vector in $\R^D$ with parameter $\sigma_t^2$.
	Assume that the Markov Chain $\{X_t : 1 \leq t \leq T\}$ on $\M^* \in 
	\mclass_\varkappa^d$ has a stationary 
	measure $\pi$ and denote a distribution of $X_t$ by $\p_t$.
	Let $Y_t = X_t + \eps_t$ and denote the convolutions of $\p_t$ and $\pi$ with the 
	distribution of $\eps_t$ by $\widetilde\p_t$ and $\widetilde\pi$ respectively.
	Then, for any $t \in \{1, \dots, T\}$, we have
	\[
	\left| \E_{\widetilde\p_t} d^2(Y_t, \M) - \E_{\widetilde\pi} d^2(Y_t, \M) \right|
	\leq (128 \sigma_t^2 D + 16R^2) \| \p_t - \pi \|_{TV}^{1/2}.
	\]
\end{Lem}

\begin{proof}
	Denote the convolutions of $\p_t$ and $\pi$ with the distribution of $\eps_t$ by 
	$\widetilde\p_t$ and $\widetilde\pi$ respectively.
	Then, due to Lemma \ref{lem1}, we have
	\[
	\left| \E_{\widetilde\p_t} d^2(Y_t, \M) - \E_{\widetilde\pi} d^2(Y_t, \M) \right|
	\leq \sqrt{\E_{\widetilde\p_t} d^4(Y_t, \M) + \E_{\widetilde\pi} d^4(Y_t, \M)} 
	\sqrt{\| \widetilde\p_t - \widetilde\pi \|_{TV}}.
	\]
	Since the total variation distance between convolutions $\widetilde\p_t$ and 
	$\widetilde\pi $ is not greater than $\|\p_t - \pi\|_{TV}$, it remains to prove
	\[
	\sqrt{\E_{\widetilde\p_t} d^4(Y_t, \M) + \E_{\widetilde\pi} d^4(Y_t, \M)}
	\leq (128 \sigma_t^2 D + 16R^2).
	\]
	
	Using the inequality $(a + b)^4 \leq 8a^4 + 8b^4$, we obtain
	\begin{align}
	\label{lem2_1}
	&
	\E_{\widetilde\p_t} d^4(Y_t, \M)
	= \E_{\widetilde\p_t} \min\limits_{x \in \M} \|X_t + \eps_t - x\|^4
	\\&\notag
	\leq \E_{\widetilde\p_t} \left( \min\limits_{x \in \M} 8\|X_t - x\|^4 + 8\|\eps_t\|^4 
	\right)
	= 8\E_{\p_t} d^4(X_t, \M) + 8\E\|\eps_t\|^4.
	\end{align}
	Note that
	\begin{equation}
	\label{lem2_2}
	\E_{\p_t} d^4(X_t, \M) \leq d_H^4(\M^*, \M) \leq (2R)^4,
	\end{equation}
	where the last inequality follows from the fact that, by definition of 
	$\mclass_\varkappa^d$, $\M$ and $\M^*$ are contained in $\B(0, R)$.
	
	Finally, consider $\E\|\eps_t\|^2$.
	\[
	\E\|\eps_t\|^2
	= \int\limits_{0}^{+\infty} \p\left( \|\eps_t\| > v^{1/4} \right) \dd v
	= \int\limits_{0}^{+\infty} \p\left( \max\limits_{\|u\|=1} u^T\eps_t > v^{1/4} \right) 
	\dd v.
	\]
	From the proof of Theorem 1.19 in \cite{r15}, we have
	\begin{align}
	\label{lem2_3}
	\E\|\eps_t\|^2
	&\notag
	= \int\limits_{0}^{+\infty} \p\left( \max\limits_{\|u\|=1} u^T\eps_t > 
	v^{1/4} \right) \dd v
	\leq \int\limits_{0}^{+\infty} \min\left\{1, 6^D e^{-\frac{\sqrt 
			v}{8\sigma_t^2}}\right\} \dd v
	\\&
	= \left( 8\sigma_t^2 D \log 6 \right)^2 + 2 \int\limits_{8\sigma_t^2 D \log 
		6}^{+\infty} 6^D e^{-\frac{v}{8\sigma_t^2}} t\dd v
	\\&\notag
	= \left( 8\sigma_t^2 D \log 6 \right)^2 + 2 \int\limits_{0}^{+\infty}
	e^{-\frac{v}{8\sigma_t^2}} (v + 8\sigma_t^2 D \log 6) \dd v
	\\&\notag
	= 64 \sigma_t^2 D^2 \log^2 6 + 128\sigma_t^4 D \log 6 + 256 \sigma_t^4
	< 2^{10} \sigma_t^4 D^2.
	\end{align}
	
	The inequalities \eqref{lem2_1}, \eqref{lem2_2} and \eqref{lem2_3} imply
	$\E_{\widetilde\p_t} d^4(Y_t, \M) \leq 2^{13} \sigma_t^4 D^2 + 2^7 R^4$.
	Similarly, one can show that $\E_{\widetilde\pi} d^4(Y_t, \M) \leq 2^{13} \sigma_t^4 
	D^2 + 2^7 R^4$.
	The inequality
	\[
	\sqrt{\E_{\widetilde\p_t} d^4(Y_t, \M) + \E_{\widetilde\pi} d^4(Y_t, \M)}
	\leq \sqrt{2^{14} \sigma_t^4 D^2 + 2^8 R^4}
	\leq 128\sigma_t^2 D + 16 R^2.
	\]
\end{proof}

\begin{Lem}
	\label{lem6}
	Assume that $\{ X_t : 1 \leq t \leq T\} \subset \M^*$ satisfies \eqref{a3}.
	Fix any $x \in \M^*$, $h < h_0$.
	Then
	\[
	\p\left( \sum\limits_{t=1}^T \1(X_t \in \B(x, h)) < \frac{p_0 h^d \lfloor T/k\rfloor}2 
	\right)
	\leq e^{-\frac{p_0h^d}{12} \lfloor T/k\rfloor},
	\]
	and
	\[
	\p\left( \sum\limits_{t=1}^T \1\left( X_t \in \B(x, h) \right) > \frac{3p_1 h^d T}2 
	\right)
	\leq e^{-\frac{p_1h^d}{12} \lceil T/k\rceil},
	\]
	
\end{Lem}

\begin{proof}
It holds that
\begin{align*}
	&
	\p\left( \sum\limits_{t=1}^T \1\left( X_t \in \B(x, h) \right) < \frac{p_0 h^d \lfloor T/k 
	\rfloor}2 \right)
	\leq \p\left( \sum\limits_{t=1}^{k\lfloor T/k\rfloor} \1\left( X_t \in \B(x, h) \right) < 
	\frac{p_0 h^d \lfloor T/k \rfloor}2 \right)
	\\&
	\leq \p\left( \sum\limits_{j=1}^{\lfloor T/k\rfloor} \sum\limits_{t=(j-1)k + 1}^{jk} \1\left( 
	X_t \in \B(x, h) \right) < \frac{p_0 h^d \lfloor T/k \rfloor}2 \right)
	\\&
	\leq \p\left( \sum\limits_{j=1}^{\lfloor T/k\rfloor} \1\left( \exists \, t \in \{(j-1)k + 1, \dots, 
	jk\} :  X_t \in \B(x, h) \right) < \frac{p_0 h^d \lfloor T/k \rfloor}2 \right)
	\\&
	\equiv \p\left( \sum\limits_{j=1}^{\lfloor T/k\rfloor} \xi_j < \frac{p_0 h^d \lfloor T/k 
	\rfloor}2 \right),
\end{align*}
where we introduced $\xi_j = \1\left( \exists \, t \in \{(j-1)k + 1, \dots, jk\} :  X_t \in \B(x, h) 
\right)$, $1 \leq j \leq \lfloor T/k \rfloor$.
Due to \eqref{a3}, we have
\[
	\E\left( \xi_j \cond \F_{(j-1)k} \right)
	\geq \max\limits_{(j-1)k \leq t \leq jk} \p\left( X_t \in \B(x, h) \cond \F_{(j-1)k} \right)
	\geq \frac1k \sum\limits_{(j-1)k \leq t \leq jk} \p\left( X_t \in \B(x, h) \cond \F_{(j-1)k} 
	\right) \geq p_0 h^d.
\]
Applying the martingale Bernstein inequality (see \cite{f75}, (1.6)), we obtain
\[
	\p\left( \sum\limits_{j=1}^{\lfloor T/k \rfloor} \xi_j < \frac{p_0 h^d}2 \left\lfloor\frac 
	Tk\right\rfloor \right)
	\leq \exp\left\{ -\frac{ \lfloor T/k\rfloor (p_0 h^d)^2/8}{p_0 h^d(1 - p_0 h^d) + p_0 h^d 
	/2} \right\}
	\leq e^{-\frac{p_0 h^d}{12} \lfloor T/k\rfloor}.
\]

Similarly,
\begin{align*}
	&
	\p\left( \sum\limits_{t=1}^T \1\left( X_t \in \B(x, h) \right) > \frac{3p_1 h^d T}2 \right)
	\leq \p\left( \sum\limits_{t=1}^{k\lceil T/k\rceil} \1\left( X_t \in \B(x, h) \right) < 
	\frac{3p_1 h^d T}2 \right)
	\\&
	\leq \p\left( \sum\limits_{j=1}^{\lceil T/k\rceil} \sum\limits_{t=(j-1)k + 1}^{jk} \1\left( 
	X_t \in \B(x, h) \right) > \frac{3 p_1 h^d T}2 \right)
	\\&
	\leq \p\left( \sum\limits_{j=1}^{\lceil T/k\rceil} \1\left( \forall \, t \in \{(j-1)k + 1, \dots, 
	jk\} :  X_t \in \B(x, h) \right) > \frac{3 p_1 h^d T}{2k} \right)
	\leq \p\left( \sum\limits_{j=1}^{\lfloor T/k\rfloor} \eta_j > \frac{3p_1 h^d \lceil T/k \rceil
	\rfloor}2 \right),
\end{align*}
where $\eta_j = \1\left( \forall \, t \in \{(j-1)k + 1, \dots, jk\} :  X_t \in \B(x, h) \right)$, $1 \leq 
j \leq \lceil T/k \rceil$.
Condition \eqref{a3} yields
\[
	\E\left( \eta_j \cond \F_{(j-1)k} \right)
	\leq \min\limits_{(j-1)k \leq t \leq jk} \p\left( X_t \in \B(x, h) \cond \F_{(j-1)k} \right)
	\leq \frac1k \sum\limits_{(j-1)k \leq t \leq jk} \p\left( X_t \in \B(x, h) \cond \F_{(j-1)k} 
	\right) \leq p_1 h^d,
\]
and, using the martingale Bernstein inequality we obtain
\[
	\p\left( \sum\limits_{j=1}^{\lceil T/k \rceil} \xi_j > \frac{3p_1 h^d}2 \left\lceil\frac 
	Tk\right\rceil \right)
	\leq \exp\left\{ -\frac{ \lfloor T/k\rfloor (p_1 h^d)^2/8}{p_1 h^d(1 - p_1 h^d) + p_1 h^d 
	/2} \right\}
	\leq e^{-\frac{p_1 h^d}{12} \lceil T/k\rceil}.
\]
\end{proof}

%% file: contents/ts_appendix_b.tex
\section{Proofs related to Theorem \ref{th1}}

\subsection{Proof of Lemma \ref{lem3}}
\label{lem3_proof}

	It holds that
	\begin{align*}
	&
	\E \sup\limits_{\M\in\mclass_\varkappa^d} \left| \frac1T 
	\sum\limits_{t=1}^T \E d^2(Y_t, \M) - \frac1T \sum\limits_{t=1}^T d^2(Y_t, \M) 
	\right|
	\\&
	\leq \E \sup\limits_{\M\in\mclass_\varkappa^d} \left| \frac1T 
	\sum\limits_{t=1}^T \overline\E d^2(Y_t, \M) - \frac1T \sum\limits_{t=1}^T d^2(Y_t, 
	\M) 
	\right|
	\\&
	+ \sup\limits_{\M\in\mclass_\varkappa^d} \left| \frac1T \sum\limits_{t=1}^T \E 
	d^2(Y_t, \M) - \frac1T \sum\limits_{t=1}^T \overline\E d^2(Y_t, \M) \right|.
	\end{align*}
	Due to Lemma \ref{lem2} and the spectral gap condition \eqref{a2}, we 
	have
	\begin{align*}
	&
	\sup\limits_{\M\in\mclass_\varkappa^d} \left| \frac1T \sum\limits_{t=1}^T \E 
	d^2(Y_t, \M) - \frac1T \sum\limits_{t=1}^T \overline\E d^2(Y_t, 
	\M) \right|
	\leq \frac1T \sup\limits_{\M\in\mclass_\varkappa^d}  \sum\limits_{t=1}^T \left| \E 
	d^2(Y_t, \M) -  \overline\E d^2(Y_t, \M) \right|
	\\&
	\leq \sum\limits_{t=1}^T \frac{A(128\sigma_t^2 D + 16 R^2)}{T} (1 - 
	\rho)^{t/2}
	\leq \frac{A(128\sigma_{\max}^2 D + 16R^2)}{T(1 - \sqrt{1 - \rho})},
	\end{align*}
	where $\sigma_{\max} = \max\limits_{1 \leq t \leq T} \sigma_t$.
	Using the inequality $1 - \sqrt{1 - \rho} \geq \rho/2$, we obtain
	\[
		\sup\limits_{\M\in\mclass_\varkappa^d} \left| \frac1T \sum\limits_{t=1}^T \E 
		d^2(Y_t, \M) - \frac1T \sum\limits_{t=1}^T \overline\E d^2(Y_t, 
		\M) \right|
		\leq \frac{2A(128\sigma_{\max}^2 D + 16R^2)}{T\rho}.
	\]
	Similarly,
	\begin{align*}
	&
	\E \sup\limits_{\M\in\mclass_\varkappa^d} \left| \frac1T 
	\sum\limits_{t=1}^T \overline\E d^2(Y_t, \M) - \frac1T \sum\limits_{t=1}^T d^2(Y_t, 
	\M) \right|
	\\&
	- \overline\E \sup\limits_{\M\in\mclass_\varkappa^d} \left| \frac1T 
	\sum\limits_{t=1}^T \E d^2(Y_t, \M) - \frac1T \sum\limits_{t=1}^T d^2(Y_t, \M) 
	\right|
	\\&
	\leq \frac1T \sum\limits_{t=1}^T \E \sup\limits_{\M\in\mclass_\varkappa^d} \left| 
	\overline\E d^2(Y_t, \M) - d^2(Y_t, \M) \right| - \overline\E 
	\sup\limits_{\M\in\mclass_\varkappa^d} \left| \overline\E d^2(Y_t, \M) - d^2(Y_t, 
	\M) \right|.
	\end{align*}
	Applying Lemma \ref{lem1} and using the inequalities
	\begin{align*}
		&
		\E \sup\limits_{\M \in \mclass_\varkappa^d} \left[ d^2(Y_t, \M) - \overline\E 
		d^2(Y_t, \M) \right]^2
		\leq \E \sup\limits_{\M \in \mclass_\varkappa^d} \left[2\|\eps_t\|^2 + 2\overline\E 
		\|\eps_t\|^2 + 4 d_H^2(\M, \M^*) \right]^2
		\\&
		\leq \E \left[2\|\eps_t\|^2 + 2\overline\E \|\eps_t\|^2 + 4 (2R)^2 \right]^2
		\leq 16 \E \|\eps_t\|^4 + 16\overline\E \|\eps_t\|^4 + 2^9 R^4
		\leq 2^{14} \sigma_t^4 D^2 + 2^9 R^4,
	\end{align*}
	where the last inequality follows from \eqref{lem2_3}, we conclude
	\begin{align*}
		&
		\frac1T \sum\limits_{t=1}^T \E \sup\limits_{\M\in\mclass_\varkappa^d} \left| 
		\overline\E d^2(Y_t, \M) - d^2(Y_t, \M) \right| - \overline\E 
		\sup\limits_{\M\in\mclass_\varkappa^d} \left| \overline\E d^2(Y_t, \M) - d^2(Y_t, 
		\M) \right|
		\\&
		\leq \sum\limits_{t=1}^T \frac{2A(128 \sigma_t^2 D + 16 R^2)}{T} (1 - 
		\rho)^{t/2}
		\leq \frac{2A(128 \sigma_{\max}^2 D + 16 R^2)}{T(1 - \sqrt{1 - \rho})}
		\leq \frac{4A(128 \sigma_{\max}^2 D + 16 R^2)}{T\rho}.
	\end{align*}
	Thus,
	\begin{align*}
		&
		\E \sup\limits_{\M\in\mclass_\varkappa^d} \left| \frac1T 
		\sum\limits_{t=1}^T \E d^2(Y_t, \M) - \frac1T \sum\limits_{t=1}^T d^2(Y_t, \M) 
		\right|
		\\&
		- \overline \E \sup\limits_{\M\in\mclass_\varkappa^d} \left| \frac1T 
		\sum\limits_{t=1}^T \overline \E d^2(Y_t, \M) - \frac1T \sum\limits_{t=1}^T 
		d^2(Y_t, \M) \right|
		\leq \frac{6A(128 \sigma_{\max}^2 D + 16 R^2)}{T\rho}.
	\end{align*}

\subsection{Proof of Lemma \ref{lem4}}
\label{lem4_proof}

	Fix any $k \in \{1, \dots, K\}$ and consider the block $B_k = \{t_1, t_2, \dots, 
	t_{N_k}\}$, where $t_1 < t_2 < \dots < t_{N_k}$.
	Let $\overline\p$ stand for the measure, corresponding to the case when $X_t$'s 
	are generated from the stationary measure $\pi$, and let $\widetilde\p$ stand for 
	the measure, corresponding to the case, when $X_t's$ are replaced by their 
	independent copies $\widetilde X_1, \dots, \widetilde X_T$. Corollary F.3.4 in 
	\cite{dmps18} and the spectral gap condition yield
	\begin{align*}
		&
		\|\overline\p - \widetilde\p\|_{TV}
		= \sup\limits_{A_1, \dots, A_{N_k}} \left| \overline\p\left( X_{t_1} \in \A_1, \dots, 
		X_{t_{N_k}} \in A_{N_k} \right) - \prod\limits_{j=1}^{N_k} \widetilde\p\left( X_{t_j} 
		\in \A_j \right) \right|
		\\&
		\leq A(1 - \rho)^K + \sup\limits_{A_1, \dots, A_{N_k}} \left| \overline\p\left( 
		X_{t_1} \in \A_1, \dots, 
		X_{t_{N_{k-1}}} \in A_{N_{k-1}} \right) \cdot \p\left( X_{t_{N_{k}}} \in A_{N_{k}} 
		\right) - \prod\limits_{j=1}^{N_k} \widetilde\p\left( X_{t_j} 
		\in \A_j \right) \right|
		\\&
		\leq A(1 - \rho)^K + \sup\limits_{A_1, \dots, A_{N_{k-1}}} \left| \overline\p\left( 
		X_{t_1} \in \A_1, \dots, 
		X_{t_{N_{k-1}}} \in A_{N_{k-1}} \right) - \prod\limits_{j=1}^{N_{k-1}} 
		\widetilde\p\left( X_{t_j} \in \A_j \right) \right|.
	\end{align*}
	Repeating the same trick $T-1$ times, we obtain
	\[
		\|\overline\p - \widetilde\p\|_{TV}
		\leq N_k A(1 - \rho)^K.
	\]
	
	Thus,
	\begin{align*}
		&
		\sum\limits_{k=1}^K \overline\E \sup\limits_{\M \in \mclass_\varkappa^d} 
		\left| \sum\limits_{t\in B_k}\left( d^2(Y_t, \M) - \overline\E d^2(Y_t, \M) \right) 
		\right|
		\\&
		\leq \sum\limits_{t=1}^T N_k A (1 - \rho)^K + \sum\limits_{k=1}^K \widetilde\E 
		\sup\limits_{\M \in 
		\mclass_\varkappa^d} \left| \sum\limits_{t\in B_k}\left( d^2(Y_t, \M) - \widetilde\E 
		d^2(Y_t, \M) \right) \right|
		\\&
		= TA (1 - \rho)^K + \sum\limits_{k=1}^K \widetilde\E \sup\limits_{\M \in 
		\mclass_\varkappa^d} \left| \sum\limits_{t\in B_k}\left( d^2(Y_t, \M) - \widetilde\E 
		d^2(Y_t, \M) \right) \right|.
	\end{align*}

\subsection{Proof of Lemma \ref{lem5}}
\label{lem5_proof}

	Let $\{\xi_t : t \in B_k\}$ be i.i.d. Rademacher random variables.
	Introduce the Rademacher complexity of the block $B_k$:
	\[
	\Rad_k(\mclass_\varkappa^d)
	= \widetilde \E \E_\xi \sup\limits_{\M \in \mclass_\varkappa^d} \left| \sum\limits_{t 
		\in B_k} \xi_t d^2(Y_t, \M) \right|.
	\]
	The standard symmetrization argument (see, for instance, \cite{gz84}) yields
	\[
	\widetilde \E \sup\limits_{\M \in \mclass_\varkappa^d} \left| \sum\limits_{t \in B_k} 
	d^2 (Y_t, \M) - \widetilde \E d^2 (Y_t, \M) \right| \leq 2 
	\Rad_k(\mclass_\varkappa^d).
	\]
	Note that, for any $\M, \M' \in \mclass_\varkappa^d$, it 
	holds
	\begin{equation}
	\label{lem3_1}
	\left|d(Y_t, \M) - d(Y_t, \M')\right|  \leq d_H(\M, \M').
	\end{equation}
	Indeed,
	\[
	d(Y_t, \M) = \min\limits_{x \in \M} \|Y_t - x\|
	\leq \|Y_t - \proj{\M'}{Y_t}\| + \min\limits_{x \in \M} \|\proj{\M'}{Y_t} - x\|
	\leq d(Y_t, \M') + d_H(\M, \M').
	\]
	Similarly, one can prove the inequality $d(Y_t, \M') \leq d(Y_t, \M') + d_H(\M, \M')$, 
	and therefore, \eqref{lem3_1} holds.
	Then
	\begin{align*}
	\left|d^2(Y_t, \M) - d^2(Y_t, \M')\right|
	&
	= \left|d(Y_t, \M) - d(Y_t, \M')\right| \left( d(Y_t, \M) + d(Y_t, \M')\right)
	\\&
	\leq 2 d_H(\M, \M') \left( \|\eps_t\| + 2R \right),
	\end{align*}
	where the last inequality holds due to the fact that
	\[
	d(Y_t, \M)
	= \min\limits_{x\in\M} \|X_t + \eps_t - x\|
	\leq \min\limits_{x\in\M} \|X_t - x\| + \|\eps_t\|
	\leq d_H(\M, \M^*) + \|\eps_t\|
	\leq 2R + \|\eps_t\|.
	\]
	Also, note that
	\begin{align*}
	\sqrt{\sum\limits_{t\in B_k} \left( d^2(Y_t, \M) - d^2(Y_t, \M') \right)^2}
	&
	\leq \sqrt{2 d_H^2(\M, \M') \sum\limits_{t\in B_k} \left( \|\eps_t\| + 2R \right)^2}
	\\&
	\leq d_H(\M, \M')  \sqrt{2 \sum\limits_{t\in B_k} \left( \|\eps_t\| + 2R 
		\right)^2}.
	\end{align*}
	Applying the chaining technique (see \cite{sst10}, Lemma A.3), we obtain the 
	following form 
	of the Dudley's integral:
	\begin{equation}
		\label{lem3_2}
		\Rad_k(\mclass_\varkappa^d)
		\lesssim \widetilde\E \left( \gamma \sum\limits_{t\in B_k} (2R + \|\eps_t\|) + 
		\sqrt{2 \sum\limits_{t\in B_k} \left( \|\eps_t\| + 2R \right)^2} 
		\int\limits_\gamma^{2R} 
		\sqrt{\log \N(\mclass_\varkappa^d, d_H, u)} \dd u \right), \quad \forall \gamma \in 
		[0, 2R],
	\end{equation}
	where $\N(\mclass_\varkappa^d, d_H, u)$ is the $u$-covering number of 
	$\mclass_\varkappa^d$ with respect to the Hausdorff distance $d_H$.
	Theorem 9 in \cite{gppvw12a} claims that
	\[
	\N(\mclass_\varkappa^d, d_H, u)
	\leq c_1 \binom{D}{d}^{(c_2/\varkappa)^D} \exp\left\{ 2^{d/2}(D - d) (c_2 / 
	\varkappa)^D u^{-d/2} \right\},
	\]
	where the constant $c_2$ depends only on $\varkappa$ and $d$.
	Using the inequality $\binom Dd \leq (eD/d)^d$, we obtain
	\[
		\log \N(\mclass_\varkappa^d, d_H, u) \leq \log c_1 + d (c_2/\varkappa)^D 
		\log\frac{eD}d + 2^{d/2}(D - d) (c_2 / \varkappa)^D u^{-d/2}.
	\]
	In \eqref{lem3_2}, choose
	\[
		\gamma =
		\begin{cases}
		0, \quad d < 4,\\
		\psi_k, \quad d = 4,\\
		\psi_k^{-4/d}, \quad d > 4.
		\end{cases},
	\]
	where 
	\[
		\psi_k = \frac{R\sqrt{D N_k} + D \sqrt{\sum\limits_{t \in B_k} \sigma_t^2}}{R N_k + 
		\sqrt{D} \sum\limits_{t \in B_k} \sigma_t}.
	\]
	Then
	\begin{equation*}
		\Rad_k(\mclass_\varkappa^d)
		\lesssim
		\begin{cases}
			D\sqrt{\sum\limits_{t\in B_k} \sigma_t^2} + R\sqrt{D N_k}, \quad d < 4,\\
			\left(D\sqrt{\sum\limits_{t\in B_k} \sigma_t^2} + R\sqrt{D N_k}\right) \log 
			\psi_k, \quad d = 4,\\
			\left(D\sqrt{\sum\limits_{t\in B_k} \sigma_t^2} + R\sqrt{D N_k}\right) 
			\psi_k^{1-4/d}, \quad d > 4.
		\end{cases}
	\end{equation*}

%% file: contents/ts_appendix_c.tex
\section{Proofs related to Theorem \ref{th2}}
\label{th2_proofs}

\subsection{Proof of Lemma \ref{lem8}}

Assume that there exists $x_0 \in \M^*$ such that $d(x_0, \M^*) > c\varkappa$.
Then, for any $x \in \B(x_0, c\varkappa/2)$, it holds that $d(x, \M) > c\varkappa/2$.
Due to Lemma \ref{lem7}, with probability at least $1 - \frac{4^d V}{(c\varkappa)^d} 
e^{-\frac{p_0 (c\varkappa/4)^d}{12} \lfloor T/k\rfloor}$,
the ball $\B(x_0, c\varkappa/2)$ contains at least $0.5 p_0(c\varkappa/4)^d \lfloor T/k 
\rfloor$ points.
On this event, for any $\M \in \mclass_\varkappa^d$, we have
\begin{align*}
&
	\frac 1T \sum\limits_{t=1}^T d^2(Y_t, \M)
	= \frac 1T \sum\limits_{t=1}^T \min\limits_{x \in \M} \|Y_t - x\|^2
	\geq \frac 1T \sum\limits_{t=1}^T \left( \min\limits_{x \in \M} \frac 12 \|X_t - x\|^2 - 
	\|\eps_t\|^2 \right)
	\\&
	= \frac 1{2T} \sum\limits_{t=1}^T d^2(X_t, \M) - \frac1T \sum\limits_{t=1}^T 
	\|\eps_t\|^2
	> \frac{p_0}{2k} \left( \frac{c\varkappa}4 \right)^{d+2} - \frac1T \sum\limits_{t=1}^T 
	\|\eps_t\|^2.
\end{align*}
On the other hand,
\[
\frac1T \sum\limits_{t=1}^T d^2(Y_t, \M^*)
\leq \frac1T \sum\limits_{t=1}^T \|\eps_t\|^2.
\]
Lemma \ref{lem_conc} implies
\[
	\frac1T \sum\limits_{t=1}^T \|\eps_t\|^2
	\leq \frac8T \sum\limits_{t=1}^T \sigma_t^2 + 64D\sigma_{\max}^2 + \frac{64}T 
	\sqrt{\sum\limits_{t=1}^T \sigma_t^4}\left( \log\frac1\delta 
	\vee \sqrt{\log\frac1\delta} \right).
\]
Choose $\delta = 1/T$.
If $T$ is large enough then, with probability at least $1 - \frac{4^d V}{(c\varkappa)^d} 
e^{-\frac{p_0 (c\varkappa/4)^d}{12} \lfloor T/k\rfloor} - 1/T$, we have
\[
	\frac1T \sum\limits_{t=1}^T 
	\|\eps_t\|^2 < \frac{p_0}{4k} \left( \frac{c\varkappa}4 \right)^{d+2},
\]
which yields
\[
	\frac 1T \sum\limits_{t=1}^T d^2(Y_t, \M) > \frac 1T \sum\limits_{t=1}^T d^2(Y_t, \M^*)
\]
simultaneously for all $\M \in \mclass_\varkappa^d$ such that $\max\limits_{x\in\M^*} 
d(x, \M) \geq c\varkappa$.

\subsection{Proof of Lemma \ref{lem7}}

Let $\N(\M^*, h)$ be the minimal $h$-net of $\M^*$ with respect to the Euclidean 
distance.
Then Lemma \ref{lem6} yields
\begin{align*}
	&
	\p\left( \exists \, x \in \M^* : \sum\limits_{t=1}^T \1\left( X_t \in \B(x, 2h) \right) < 
	\frac{p_0 h^d}2 \left\lfloor\frac Tk \right\rfloor \right)
	\\&
	\leq \p\left( \exists \, x \in \N(\M^*, h) : \sum\limits_{t=1}^T \1\left( X_t \in \B(x, h) 
	\right) 
	< \frac{p_0 h^d}2 \left\lfloor\frac Tk \right\rfloor \right)
	\leq |\N(\M^*, h)| e^{-\frac{p_0 h^d}{12} \lfloor T/k\rfloor}.
\end{align*}
Similarly,
\begin{align*}
	&
	\p\left( \exists \, x \in \M^* : \sum\limits_{t=1}^T \1\left( X_t \in \B(x, h/2) \right) >
	\frac{3p_1 h^d T}2\right)
	\\&
	\leq \p\left( \exists \, x \in \N(\M^*, h) : \sum\limits_{t=1}^T \1\left( X_t \in \B(x, h) 
	\right) 
	< \frac{3p_1 h^d T}2 \right)
	\leq |\N(\M^*, h)| e^{-\frac{p_1 h^d}{12} \lceil T/k\rceil}.
\end{align*}
To finish the proof, note that Lemma 2.5 in \cite{blw18} implies that, for any $h \leq 2 
\varkappa$ a Euclidean ball $\B(x, h)$, $x\in \M^*$, contains a ball of radius $h/2$ with 
respect to the geodesic distance on $\M^*$.
Since the volume of $\M^*$ is at most $V$, it can be covered with $(2^d V)/h^d$ 
Euclidean balls of radius $h$.

\subsection{Proof of Lemma \ref{lem_conc}}

{\bf Proof of (1).}\\
Fix any $t \in \{1, \dots, T\}$.
Theorem 1.19 in \cite{r15} implies that, with probability at least $1 - \delta/T$, we 
have
\[
	\|\eps_t^\parallel\| \leq 4\sigma_t \sqrt{d} + 2\sigma_t \sqrt{2 \log(T/\delta)}.
\]
The union bound yields the assertion of the lemma.

\medskip
\noindent
{\bf Proof of (2).}\\

Using the $\eps$-net argument (see \cite{r15}, Theorem 1.19), we obtain
\[
	\sum\limits_{t=1}^T \|\eps_t^\perp\|
	= \max\limits_{u_1, \dots, u_T \in \B(0, 1)} \sum\limits_{t=1}^T u_t^T\eps_t^\perp
	\leq 2 \max\limits_{u_1 \in \N_{1/2}^1, \dots, u_T \in \N_{1/2}^T} \sum\limits_{t=1}^T 
	u_t^T\eps_t^\perp,
\]
where $\N_{1/2}^t$, $1 \leq t \leq T$, is the minimal $1/2$-net of $\B(0, 1) \cap 
\T_{X_t}^\perp\M^*$.
It is known that $|\N_{1/2}^t| \leq 6^{D-d}$.
For any $t$ and any $u_t \in \N_{1/2}$, $u_t^T\eps_t^\perp$ is a sub-Gaussian random 
variable with parameter $\sigma_t^2$.
The Hoeffding inequality and the union bound yield that, for any $u > 0$,
\[
	\p\left( \sum\limits_{t=1}^T \|\eps_t^\perp\| \geq u \right) \leq 6^{(D-d)T} \exp\left\{ 
	-\frac{u^2}{8 \sum\limits_{t=1}^T \sigma_t^2} \right\},
\]
and the claim of the lemma follows.

\medskip
\noindent
{\bf Proof of (3).}\\

The $\eps$-net argument (see \cite{r15}, Theorem 1.19) yields
\[
	\|\eps_t^\parallel\|^2 = \max\limits_{u \in \B(0,1)} (u^T\eps_t^\parallel)^2
	\leq 4 \max\limits_{u \in \N_{1/2}^t} (u^T\eps_t^\parallel)^2,
\]
where $\N_{1/2}^t$, $1 \leq t \leq T$, is the minimal $1/2$-net of $\B(0, 1) \cap 
\T_{X_t}^\perp\M^*$.
Then
\begin{align*}
	\frac1T \sum\limits_{t=1}^T \|\eps_t^\parallel\|^2
	= \frac4T \max\limits_{u_1 \in \N_{1/2}^1, \dots, u_T \in \N_{1/2}^T} \sum\limits_{t=1}^T 
	(u_t^T\eps_t^\parallel)^2.
\end{align*}
Lemma 1.12 in \cite{r15} implies that $(u_t^T\eps_t^\parallel)^2$ is sub-Exponential 
random variable $\text{SE}(16\sigma_t^2, 16\sigma_t^2)$ (see \cite{w19}, Definition 2.7 
for definition of sub-exponential random variable).
Then, for any $u_1 \in \N_{1/2}^1, \dots, u_T \in \N_{1/2}^T$, $\sum\limits_{t=1}^T 
(u_t^T\eps_t^\parallel)^2$ is sub-Exponential random variable 
$\text{SE}\left(16\sqrt{\sum\limits_{t=1}^T \sigma_t^4}, 16\sigma_{\max}^2 \right)$.
Using a concentration inequality for sub-exponential random variables (see \cite{w19}, 
Proposition 2.9), we obtain that
\[
	\frac4T \sum\limits_{t=1}^T (u_t^T\eps_t^\parallel)^2
	\leq \frac4T \sum\limits_{t=1}^T \sigma_t^2 + \frac{32}T \left( \sigma_{\max}^2 
	\log\frac1\delta \vee \sqrt{\sum\limits_{t=1}^T \sigma_t^4 \log\frac1\delta} \right)
\]
with probability at least $1 - \delta$.
The union bound implies that, with probability at least $1 - \delta$, it holds that
\begin{align*}
	&
	\frac4T \max\limits_{u_1 \in \N_{1/2}^1, \dots, u_T \in \N_{1/2}^T} \sum\limits_{t=1}^T 
	(u_t^T\eps_t^\parallel)^2
	\\&
	\leq \frac4T \sum\limits_{t=1}^T \sigma_t^2 + \frac{32}T \left( \sigma_{\max}^2 
	\left( dT\log 6 + \log\frac1\delta\right) \vee \sqrt{\left(dT\log 
	6 + \log\frac1\delta \right) \sum\limits_{t=1}^T \sigma_t^4 } \right)
	\\&
	\leq \frac4T \sum\limits_{t=1}^T \sigma_t^2 + 64d \sigma_{\max}^2 + \frac{32}T 
	\sqrt{\sum\limits_{t=1}^T \sigma_t^4}\left( \log\frac1\delta 
	\vee \sqrt{\log\frac1\delta} \right).
\end{align*}
\medskip
\noindent
{\bf Proof of (4).}\\
The proof of (4) is similar to the proof of (3).

\medskip
\noindent
{\bf Proof of (5).}\\

Using the large deviation inequality for the norm of sub-Gaussian random vector (see 
\cite{r15}, the proof of Theorem 1.19), we 
obtain $\p\left(\|\eps_t^\parallel\| \geq c\varkappa\right) \leq 6^d 
e^{-\frac{c^2\varkappa^2}{4\sigma_t^2}}$.
The Bernstein's inequality for Bernoulli random variables yields that, for any $\delta \in 
(0, 1)$, with probability at least $1 - \delta$, it holds that
\begin{align*}
	\frac1T \sum\limits_{t=1}^T \1\left( \|\eps_t^\parallel\| \geq c\varkappa \right)
	&
	\leq \frac{6^d}T \sum\limits_{t=1}^T e^{-\frac{c^2\varkappa^2}{4\sigma_t^2}} + 
	\sqrt{\frac{2 \cdot 6^d\log(1/\delta)}T \sum\limits_{t=1}^T 
	e^{-\frac{c^2\varkappa^2}{4\sigma_t^2}}} + \frac{2\log(1/\delta)}{3T}
	\\&
	< \frac{2 \cdot 6^d}T \sum\limits_{t=1}^T e^{-\frac{c^2\varkappa^2}{4\sigma_t^2}} + 
	\frac{2\log(1/\delta)}{T}.
\end{align*}

%% file: contents/ts_appendix_d.tex
\section{Plots of the predictions}
\label{plots}

\begin{figure}[h!]%
	\centering
	\subfloat[][\centering One day 
	ahead]{{\includegraphics[scale=0.35]{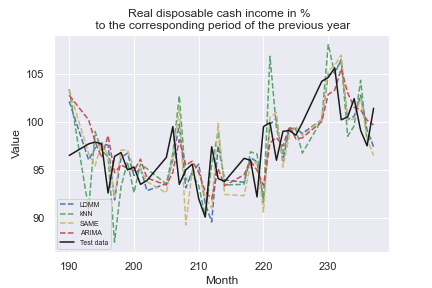}
	}}%
	\qquad
	\subfloat[][\centering Two days 
	ahead]{{\includegraphics[scale=0.35]{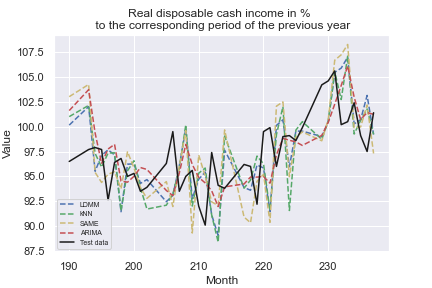}
	}}%
	\qquad
	\subfloat[][\centering Three days 
	ahead]{{\includegraphics[scale=0.35]{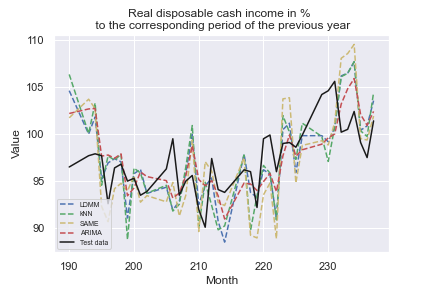}
	}}%
	\qquad
	\subfloat[][\centering Four days 
	ahead]{{\includegraphics[scale=0.35]{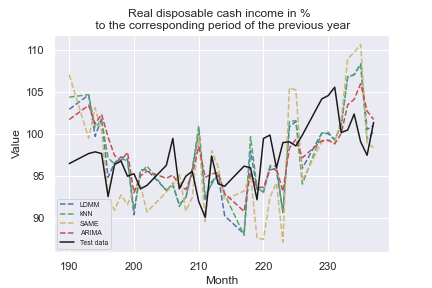}
	}}%
	\caption{Prediction of the first component of multidimensional time series}
	\label{fig_1component}
\end{figure}

\begin{figure}[h!]%
	\centering
	\subfloat[][\centering One day 
	ahead]{{\includegraphics[scale=0.35]{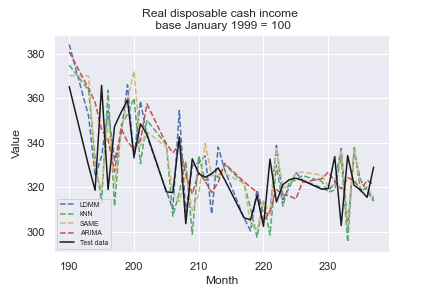}
	}}%
	\qquad
	\subfloat[][\centering Two days 
	ahead]{{\includegraphics[scale=0.35]{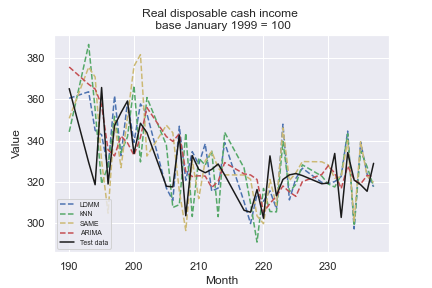}
	}}%
	\qquad
	\subfloat[][\centering Three days 
	ahead]{{\includegraphics[scale=0.35]{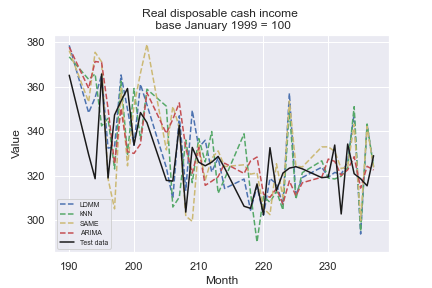}
	}}%
	\qquad
	\subfloat[][\centering Four days 
	ahead]{{\includegraphics[scale=0.35]{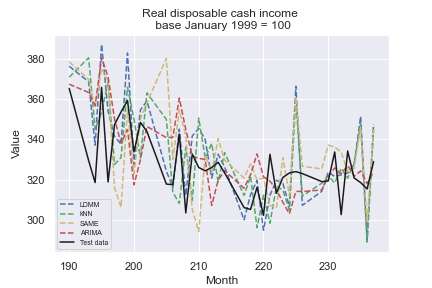}
	}}%
	\caption{Prediction of the second component of multidimensional time series}
\end{figure}

\begin{figure}[h!]%
	\centering
	\subfloat[][\centering One day 
	ahead]{{\includegraphics[scale=0.35]{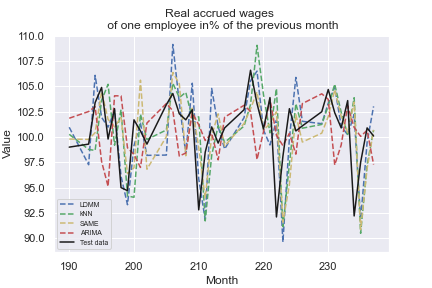}
	}}%
	\qquad
	\subfloat[][\centering Two days 
	ahead]{{\includegraphics[scale=0.35]{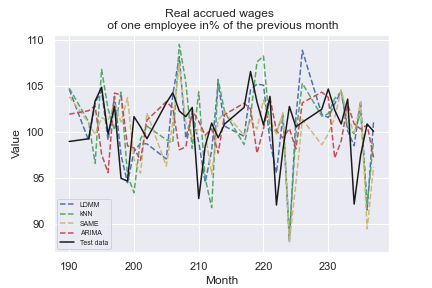}
	}}%
	\qquad
	\subfloat[][\centering Three days 
	ahead]{{\includegraphics[scale=0.35]{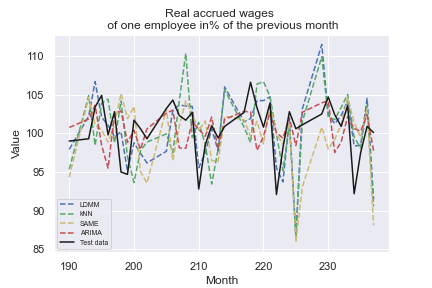}
	}}%
	\qquad
	\subfloat[][\centering Four days 
	ahead]{{\includegraphics[scale=0.35]{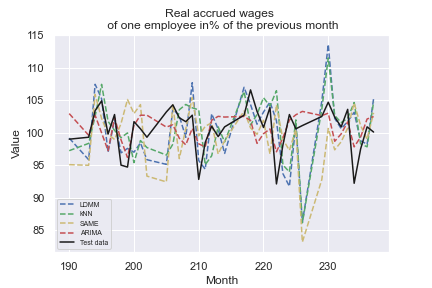}
	}}%
	\caption{Prediction of the third component of multidimensional time series}
\end{figure}

\begin{figure}[h!]%
	\centering
	\subfloat[][\centering One day 
	ahead]{{\includegraphics[scale=0.35]{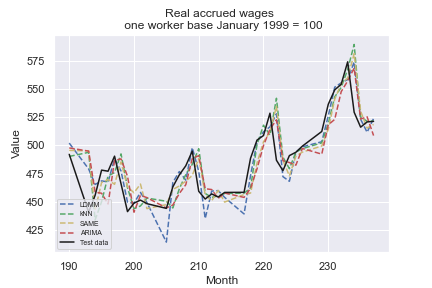}
	}}%
	\qquad
	\subfloat[][\centering Two days 
	ahead]{{\includegraphics[scale=0.35]{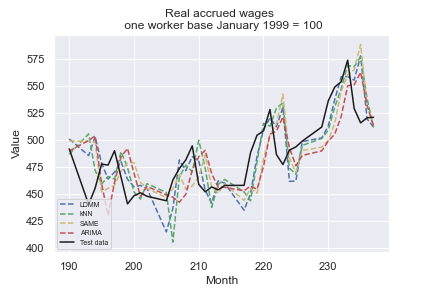}
	}}%
	\qquad
	\subfloat[][\centering Three days 
	ahead]{{\includegraphics[scale=0.35]{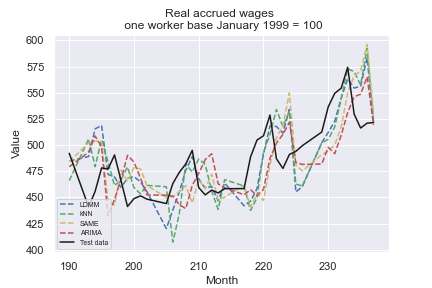}
	}}%
	\qquad
	\subfloat[][\centering Four days 
	ahead]{{\includegraphics[scale=0.35]{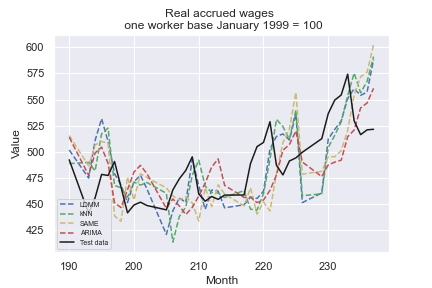}
	}}%
	\caption{Prediction of the fourth component of multidimensional time series}
	\label{fig_4component}
\end{figure}